\documentclass[12pt, reqno]{amsart}
\usepackage{amssymb,latexsym,amsmath,amscd,amsfonts}
\usepackage{latexsym}
\usepackage[mathscr]{eucal}
\usepackage{bm}

\newcommand{\ft}{\mathcal F_O}

\def ~{\hspace{1mm}}

\def\textmatrix#1&#2\\#3&#4\\{\bigl({#1 \atop #3}\ {#2 \atop #4}\bigr)}
\def\dispmatrix#1&#2\\#3&#4\\{\left({#1 \atop #3}\ {#2 \atop #4}\right)}
\newcommand{\beg}{\begin{equation}}
\newcommand{\eeg}{\end{equation}}
\newcommand{\ben}{\begin{eqnarray*}}
\newcommand{\een}{\end{eqnarray*}}

\newtheorem{thm}{Theorem}[section]
\newtheorem{cor}[thm]{Corollary}
\newtheorem{lem}[thm]{Lemma}

\newtheorem{prop}[thm]{Proposition}
\numberwithin{equation}{section} \theoremstyle{definition}
\newtheorem{defn}[thm]{Definition}

\newtheorem{eg}[thm]{Example}

\def\textmatrix#1&#2\\#3&#4\\{\bigl({#1 \atop #3}\ {#2 \atop #4}\bigr)}
\def\dispmatrix#1&#2\\#3&#4\\{\left({#1 \atop #3}\ {#2 \atop #4}\right)}

\begin{document}
\title[Structure theorems for $\Gamma_n$-contractions and $\mathbb E$-contractions]
{Structure theorems for operators associated with two domains
related to $\mu$-synthesis}

\author[Bappa Bisai]{Bappa Bisai}
\address[Bappa Bisai]{Mathematics Department, Indian Institute of Technology Bombay,
Powai, Mumbai - 400076, India.} \email{bisai@math.iitb.ac.in}

\author[Sourav Pal]{Sourav Pal}
\address[Sourav Pal]{Mathematics Department, Indian Institute of Technology Bombay,
Powai, Mumbai - 400076, India.} \email{sourav@math.iitb.ac.in}



\thanks{The first by named author is supported by a Ph.D fellowship
of the University Grants Commissoin (UGC). The second named author
is supported by the Seed Grant of IIT Bombay, the CPDA and the
INSPIRE Faculty Award (Award No. DST/INSPIRE/04/2014/001462) of
DST, India.}

\begin{abstract}
A commuting tuple of $n$ operators $(S_1, \dots, S_{n-1}, P)$
defined on a Hilbert space $\mathcal{H}$, for which the closed
symmetrized polydisc
\[
    \Gamma_n = \left\{ \left(\sum_{i=1}^{n}z_i,
    \sum\limits_{1\leq i<j\leq n}z_iz_j, \dots, \prod_{i=1}^{n}z_i \right) : |z_i|\leq 1, i=1, \dots, n
    \right\}
\]
is a spectral set is called a $\Gamma_n$-contraction. Also a triple of commuting operators $(A,B,P)$ for which the closed tetrablock $\overline{\mathbb E}$ is a spectral set is called an $\mathbb E$-contraction, where
\[
\mathbb E = \{ (x_1,x_2,x_3)\in\mathbb C^3\,:\,
 1-zx_1-wx_2+zwx_3  \neq 0 \quad \forall z, w \in \overline{\mathbb D} \}.
\]
There are several decomposition theorems for contraction operators in the literature due to Sz. Nagy, Foias, Levan, Kubrusly, Foguel and few others which reveal structural information of a contraction. In this article, we obtain analogues of six such major theorems for both $\Gamma_n$-contractions and $\mathbb E$-contractions. In each of these decomposition theorems, the underlying Hilbert space admits a unique orthogonal decomposition which is provided by the last component $P$. The central role in determining the structure of a $\Gamma_n$-contraction or an $\mathbb E$-contraction is played by positivity of some certain operator pencils and the existence of a unique operator tuple associated with a $\Gamma_n$-contraction or an $\mathbb E$-contraction.

\end{abstract}

\maketitle

\tableofcontents

\section{Introduction}

Throughout the paper all operators are bounded linear
transformations defined on complex Hilbert spaces. The basic
definitions are given in Section \ref{definitions}.\\

This article is devoted to studying structure theory of operators
associated with two popular domains namely the symmetrized
polydisc $\mathbb G_n$ and the tetrablock $\mathbb E$ which are
the following sets:

\begin{align*}
\mathbb G_n &=\left\{ \left(\sum_{1\leq i\leq n} z_i,\sum_{1\leq
i<j\leq n}z_iz_j,\dots,
\prod_{i=1}^n z_i \right): \,|z_i|< 1, i=1,\dots,n \right \} \subset \mathbb C^n, \\
\mathbb E &= \{ (x_1,x_2,x_3)\in\mathbb C^3\,:\,
 1-zx_1-wx_2+zwx_3  \neq 0 \quad \forall z, w \in \overline{\mathbb D} \}.
\end{align*}

These domains are closely related to the $\mu$-synthesis problem.
The $\mu$-synthesis is a part of the theory of robust control of
systems comprising interconnected electronic devices whose outputs
are linearly dependent on the inputs. Given a linear subspace $E$
of $\mathcal M_n(\mathbb C)$, the space of all $n \times n$
complex matrices, the functional
\[
\mu_E(A):= (\text{inf} \{ \|X \|: X\in E \text{ and } (I-AX)
\text{ is singular } \})^{-1}, \; A\in \mathcal M_n(\mathbb C),
\]
is called a \textit{a structured singular value}, where the linear
subspace $E$ is referred to as the \textit{structure}. If
$E=\mathcal M_n(\mathbb C)$, then $\mu_E (A)$ is equal to the
operator norm $\|A\|$, while if $E$ is the space of all scalar multiples of the identity matrix, then $\mu_E(A)$
is the spectral radius $r(A)$. For any linear subspace $E$ of
$\mathcal M_n(\mathbb C)$ that contains the identity matrix $I$,
$r(A)\leq \mu_E(A) \leq \|A\|$. We refer readers to the pioneering
work of Doyle \cite{Doyle} for the control-theory motivations
behind $\mu_E$ and for further details an interested reader can
see \cite{Francis}. The aim of $\mu$-synthesis is to find an
analytic function $F$ from the open unit disk $\mathbb D$ (with center at the origin) of the complex plane
to $\mathcal M_n(\mathbb C)$ subject to a finite
number of interpolation conditions such that $\mu_E(F(\lambda))<1$
for all $\lambda \in \mathbb D$. If $E$ is the linear subspace of
$2 \times 2$ diagonal matrices, then for any $A=(a_{ij}) \in
\mathcal M_2 (\mathbb C)$, $\mu_E (A)<1$ if and only if
$(a_{11},a_{22}, \det A)\in \mathbb E$ (\cite{A:W:Y}, Section-9).
Also if $E=\{ \lambda I:\lambda \in \mathbb C \} \subseteq
\mathcal M_n (\mathbb C)$, then $\mu_E (A)=r(A)<1$ if and only if
$\pi_n(\lambda_1,\dots, \lambda_n) \in \mathbb G_n$ (see
\cite{costara1}), where $\lambda_1, \dots , \lambda_n$ are
eigenvalues of $A$ and $\pi_n$ is the symmetrization map on
$\mathbb C^n$ defined by
\[
\pi_n(z_1,\dots, z_n) = \left(\sum_{1\leq i\leq n} z_i,\sum_{1\leq
i<j\leq n}z_iz_j,\dots, \prod_{i=1}^n z_i \right).
\]
The sets $\overline {\mathbb G_n}$ and $\overline{\mathbb E}$ are
not convex but polynomially convex. In spite of having origin in
control engineering, the domains $\mathbb G_n$ and $\mathbb E$
have been extensively studied in past two decades by numerous
mathematicians for aspects of complex geometry, function theory
and operator theory. An interested reader is referred to some
exciting works of recent past \cite{A:W:Y, ALY12, ay-jfa,
bharali1, costara, costara1, edi-zwo, jarnicki, L:K2, L:K1,
PalSubvarieties, young, Zwo} and references there in. In this
article, we contribute to the existing rich operator theory
(see \cite{T:B, tirtha-sourav, tirtha-sourav1, S:B,
sourav, S:PRational, S:P, pal-shalit} and articles referred there)
of these two domains. Operator theory on a domain is always of
independent interests, yet we expect that the results obtained in
this article will throw new lights to the function theory and
complex geometry of $\mathbb G_n$ and $\mathbb E$ which may help
the control engineering community as well. So, our primary object
of study is an operator tuple for which $\overline{\mathbb G_n}$
or $\overline{\mathbb E}$ is a spectral set.

\begin{defn}
A compact set $K\subset \mathbb C^n$ is said to be a spectral set
for a commuting $n$-tuple of operators $\underline{T}=(T_1,\dots,
T_n)$ if the Taylor joint spectrum $\sigma (\underline{T})$ of
$\underline{T}$ is a subset of $K$ and von Neumann inequality
holds for every rational function, that is,
\[
\|f(\underline{T})\|\leq \|f\|_{\infty, K} = \text{sup}\{|f(z_1,
\dots, z_n)| : (z_1, \dots, z_n) \in K\} \,,
\]
for all rational functions $f$ with poles off $K$.
\end{defn}

\begin{defn}
A commuting $n$-tuple of operators $(S_1,\dots, S_{n-1},P)$ for which
the closed symmetrized polydisc $\Gamma_n$ ($=\overline{\mathbb
G_n}$) is a spectral set is called a
$\Gamma_n$-\textit{contraction}. Similarly a commuting triple of
operators $(A,B,P)$ for which $\overline{\mathbb E}$ is a spectral
set is called an $\mathbb E$-\textit{contraction} or a
\textit{tetrablock-contraction}.
\end{defn}

It is evident from the definitions that the adjoint of a
$\Gamma_n$-contraction or an $\mathbb E$-contraction is also a
$\Gamma_n$-contraction or an $\mathbb E$-contraction respectively
and if $(S_1,\dots, S_{n-1},P)$ is a $\Gamma_n$-contraction or
$(A,B,P)$ is an $\mathbb E$-contraction, then $P$ is a
contraction.\\

One of the most wonderful discoveries in operator theory is the
canonical decomposition of a contraction due to Nagy and Foias
\cite{Nagy} which states the following:

\begin{thm}\label{thm:candecomp}
Let $T$ on $\mathcal H$ be a contraction. Let $\mathcal H_1$ be
the maximal subspace of $\mathcal H$ which reduces $T$ and on
which $T$ is unitary. Let $\mathcal H_2=\mathcal H\ominus \mathcal
H_1$. Then $T=T_1 \oplus T_2$ with respect to $\mathcal H=\mathcal
H_1 \oplus \mathcal H_2$ is the unique orthogonal decomposition of
$T$ into unitary $T_1=T|_{\mathcal H_1}$ and c.n.u
$T_2=T|_{\mathcal H_2}$. Anyone of $\mathcal H_1, \mathcal H_2$
may be equal to the trivial subspace $\{ 0 \}$.
\end{thm}

In \cite{N:L}, Levan took an appealing next step to split further
a c.n.u (completely non-unitary) contraction into two orthogonal parts of which one is a
c.n.i (completely non-isometry) contraction and the other is a unilateral
shift, i.e., a pure isometry.\\

In \cite{S:P}, the second named author of this article established
that every $\mathbb E$-contraction $(A,B,P)$ defined on $\mathcal
H$ admits an analogous orthogonal decomposition into an $\mathbb
E$-unitary $(A_1,B_1,P_1)$ and a c.n.u $\mathbb E$-contraction
$(A_2,B_2,P_2)$ with respect to the decomposition $\mathcal H
=\mathcal H_1 \oplus \mathcal H_2$ whence $P=P_1\oplus P_2$ on
$\mathcal H_1 \oplus \mathcal H_2$ is the canonical decomposition
of the contraction $P$. The beauty of this decomposition is that
the canonical decomposition of $P$ ($P=P_1\oplus P_2$ with respect
to $\mathcal H=\mathcal H_1 \oplus \mathcal H_2$) is chosen first
and then it was shown that $A, B$ reduce both $\mathcal H_1$ and
$\mathcal H_2$. In \cite{sourav3}, the same author proved that a
similar decomposition was possible for a $\Gamma_n$-contraction
too. In Theorem \ref{cd1}, we show that for a c.n.u
$\Gamma_n$-contraction $(S_1,\dots, S_{n-1},P)$ on $\mathcal H$,
if $P=P_1\oplus P_2$ with respect to $\mathcal H=\mathcal
H_1\oplus \mathcal H_2$ is the unique decomposition into
unilateral shift $P_1$ and c.n.i $P_2$, then $S_1,\dots, S_{n-1}$
also reduce both $\mathcal H_1, \mathcal H_2$ provided either
$P^*$ commutes with each $S_i$ or $\mathcal H_1$ is the maximal
invariant subspace for $P$ on which $P$ is isometry. In such cases
$(S_1,\dots, S_{n-1},P)$ admits a unique decomposition into a pure
$\Gamma_n$-isometry and a c.n.i $\Gamma_n$-contraction. So, in a
word this can be thought of as an analogue of Levan's
decomposition for the $\Gamma_n$-contractions. We also show by
example that at least one of the conditions that either $P^*$
commutes with each $S_i$ or $\mathcal H_1$ is the maximal
invariant subspace for $P$ on which $P$ is isometry is essential
for such a decomposition. Example \ref{example1} shows that if we
drop either of the conditions then we may not reach the conclusion
of Theorem \ref{cd1}. In Theorem \ref{cd11}, we show that a
similar Levan type decomposition holds for an $\mathbb
E$-contraction under similar conditions. In an analogous way, a
counter example of $\mathbb E$-contraction is given in Example
\ref{example12} to establish that we cannot ignore both the
conditions simultaneously. We used the positivity of certain
operator pencils for Levan's decomposition of a
$\Gamma_n$-contraction whereas the existence of fundamental
operator pair guarantees the similar decomposition of an $\mathbb
E$-contraction.\\

After Nagy-Foias and Levan, Kubrusly, Foguel and few other
mathematicians found different decompositions of a contraction. We
recall six such decompositions and segregate them in Section
\ref{gamma-n}. We present analogues of these decompositions for
$\Gamma_n$ contractions in the same section. Also we obtain
similar decomposition theorems for an $\mathbb E$-contraction in
Section \ref{tetra-block}. We show that in each theorem, a
$\Gamma_n$-contraction $(S_1,\dots, S_{n-1},P)$ or an $\mathbb
E$-contraction $(A,B,P)$ is split according to the corresponding
decomposition of the component $P$. In Section \ref{definitions},
we describe brief literatures of a $\Gamma_n$-contraction and an
$\mathbb E$-contraction and provide several definitions with
proper motivations. In Section \ref{results}, we accumulate from
the literature few results about $\Gamma_n$-contractions and
$\mathbb E$-contractions which we shall use in sequel.

\section{A brief literature and definitions}\label{definitions}

In this section, we recall from literature a few special classes
of $\Gamma_n$-contractions and $\mathbb E$-contractions. Also we
shall define several new classes of $\Gamma_n$-contractions and
$\mathbb E$-contractions and give proper motivations behind such
definitions. We begin with a survey of several important classes
of Hilbert space contractions, e.g., unitary, isometry,
co-isometry. It is well-known that an operator $T$, defined on a
Hilbert space $\mathcal H$, is a unitary or an isometry if it
satisfies $T^*T=TT^*=I$ or $T^*T=I$ respectively. Also a
co-isometry is the adjoint of an isometry. The following is the
geometric way of describing these special classes.

\begin{defn}
An operator $T$ on $\mathcal H$ is
\begin{itemize}
\item[(i)] a unitary if $T$ is a normal operator and the spectrum
$\sigma (T)$ of $T$ is a subset of the unit circle $\mathbb T$ ;

\item[(ii)] an isometry if it is the restriction of a unitary to
an invariant subspace, that is, if there is a Hilbert space
$\mathcal K$ that contains $\mathcal H$ as a closed linear
subspace and a unitary $U$ on $\mathcal K$ such that $\mathcal H$
is an invariant subspace for $U$ and that $U|_{\mathcal H}=T$ ;

\item[(iii)] a co-isometry if $T^*$ is an isometry.

\end{itemize}

\end{defn}

We have already witnessed that a contraction is an operator for
which $\overline{\mathbb D}$ is a spectral set. So, a unitary is a
normal contraction that has $\mathbb T$, i.e., the boundary of
$\overline{\mathbb D}$ as a spectral set. For a domain $G$ in
$\mathbb C^n$ for $n\geq 2$, the notion of boundary is generalized
as the distinguished boundary of $\overline{G}$. For a compact
subset $X$ of $\mathbb C^n$ let $\mathcal A(X)$ be an algebra of
continuous complex-valued functions on $X$ which separates the
points of $X$. A \textit{boundary} for $\mathcal A(X)$ is a closed
subset $\Delta X$ of $X$ such that every function in $\mathcal
A(X)$ attains its maximum modulus on $\Delta X$. It follows from
the theory of uniform algebras that the intersection of all the
boundaries of $X$ is also a boundary for $\mathcal A(X)$ (see
Theorem 9.1 of \cite{wermer}). This smallest boundary is called
the $\check{\textup{S}}$\textit{ilov boundary} for $\mathcal
A(X)$. When $\mathcal A(X)$ is the algebra of rational functions
which are continuous on $X$, the $\check{\textup{S}}$\textit{ilov
boundary} for $\mathcal A(X)$ is called the \textit{distinguished
boundary} of $X$ and is denoted by $bX$.\\

It is well-known that the distinguished boundary of the closed
polydisc $\overline{\mathbb D^n}$ is the $n$-torus $\mathbb T^n$.
We obtain from the literature (see \cite{edi-zwo}) that the
distinguished boundary of $\Gamma_n$ is the following set:
\begin{align*}
b\Gamma_n & = \left\{ \left(\sum_{i=1}^{n}z_i,
    \sum\limits_{1\leq i<j\leq n}z_iz_j, \dots, \prod_{i=1}^{n}z_i \right) : |z_i|= 1, i=1, \dots, n
    \right\} \\
    & = \pi_n(\mathbb T^n) \\
    & = \{ (s_1,\dots,s_{n-1},p)\in\Gamma_n\,:\, |p|=1 \}.
\end{align*}
Also the distinguished boundary of $\overline{\mathbb E}$ was
determined in \cite{A:W:Y} to be the set
\[
b\mathbb E  = \{ (a,b,p)\in \overline{\mathbb E}\,:\, |p|=1 \}.
\]

The notion of distinguished boundary naturally leads to the
following definitions which are already there in the literature of
$\Gamma_n$-contractions (see \cite{S:B}).

\begin{defn}
Let $S_1,\dots, S_{n-1},P$ be commuting operators on $\mathcal H$.
Then $(S_1,\dots, S_{n-1},P)$ is called
\begin{itemize}
\item[(i)] a $\Gamma_n$-\textit{unitary} if $S_1,\dots, S_{n-1},P$
are normal operators and the Taylor joint spectrum $\sigma_T
(S_1,\dots, S_{n-1},P)$ is a subset of $b\Gamma_n$ ;

\item[(ii)] a $\Gamma_n$-isometry if there exists a Hilbert space
$\mathcal K \supseteq \mathcal H$ and a $\Gamma_n$-unitary
$(T_1,\dots,T_{n-1},U)$ on $\mathcal K$ such that $\mathcal H$ is
a joint invariant subspace of $S_1,\dots, S_{n-1},P$ and that
$(T_1|_{\mathcal H},\dots, T_{n-1}|_{\mathcal H},U|_{\mathcal
H})=(S_1,\dots, S_{n-1},P)$ ;

\item[(iii)] a $\Gamma_n$-co-isometry if the adjoint
$(S_1^*,\dots, S_{n-1}^*,P^*)$ is a $\Gamma_n$-isometry.
\end{itemize}

\end{defn}

Also we obtain from the literature (\cite{T:B}) the following
analogous classes of $\mathbb E$-contractions.

\begin{defn}
Let $A,B,P$ be commuting operators on $\mathcal H$. Then $(A,B,P)$
is called
\begin{itemize}
\item[(i)] a $\mathbb E$-\textit{unitary} if $A,B,P$ are normal
operators and the Taylor joint spectrum $\sigma_T (A,B,P)$ is a
subset of $b\Gamma_n$ ;

\item[(ii)] a $\mathbb E$-isometry if there exists a Hilbert space
$\mathcal K \supseteq \mathcal H$ and a $\mathbb E$-unitary
$(Q_1,Q_2,V)$ on $\mathcal K$ such that $\mathcal H$ is a joint
invariant subspace of $A,B,P$ and that $(Q_1|_{\mathcal H},
Q_2|_{\mathcal H},V|_{\mathcal H})=(A,B,P)$ ;

\item[(iii)] a $\mathbb E$-co-isometry if the adjoint $(A^*,
B^*,P^*)$ is a $\mathbb E$-isometry.
\end{itemize}

\end{defn}

The following theorems from \cite{sourav14} provide clear
descriptions of a $\Gamma_n$-unitary and a $\Gamma_n$-isometry.

\begin{thm}[\cite{sourav14}, Theorems 4.2 $\&$
4.4]\label{thm:gamma-ui} A commuting tuple of operators
$(S_1,\dots,S_{n-1},P)$ is a $\Gamma_n$-unitary $($or, a
$\Gamma_n$-isometry$)$ if and only if $(S_1,\dots,S_{n-1},P)$ is a
$\Gamma_n$-contraction and $P$ is a unitary $($isometry$)$.
\end{thm}

Needless to mention that $(S_1,\dots,S_{n-1},P)$ is a
$\Gamma_n$-co-isometry if and only if $(S_1,\dots,S_{n-1},P)$ is a
$\Gamma_n$-contraction and $P$ is a co-isometry. So, it is evident
that the nature of a $\Gamma_n$-contraction
$(S_1,\dots,S_{n-1},P)$ is highly influenced by the nature of its
last component $P$. In \cite{sourav3}, the second named author of
this paper had shown that for a given $\Gamma_n$-contraction
$(S_1,\dots,S_{n-1},P)$ on $\mathcal H$, if $P=P|_{\mathcal
H_1}\oplus P|_{\mathcal H_2}$ is the canonical decomposition of
the contraction $P$ as in Theorem \ref{thm:candecomp} with respect
to $\mathcal H= \mathcal H_1 \oplus \mathcal H_2$ , then both
$\mathcal H_1, \mathcal H_2$ reduce $S_1,\dots,S_{n-1}$ and
$(S_1|_{\mathcal H_1},\dots,S_{n-1}|_{\mathcal H_1},P|_{\mathcal
H_1})$ is a $\Gamma_n$-unitary whereas $(S_1|_{\mathcal
H_2},\dots,S_{n-1}|_{\mathcal H_2},P|_{\mathcal H_2})$ is a
$\Gamma_n$-contraction for which $P|_{\mathcal H_2}$ is a c.n.u
contraction. This unique decomposition was named the
``\textit{canonical decomposition}" of a $\Gamma_n$-contraction
(see Theorem \ref{mainthm} in this paper). This naturally
motivated the author to define a c.n.u $\Gamma_n$-contraction to
be a $\Gamma_n$-contraction $(S_1,\dots, S_{n-1},P)$ for which $P$
is a c.n.u contraction and indeed such a definition is justified.
Taking cue from such dominant roles of $P$ in determining the
special classes of a $\Gamma_n$-contraction
$(S_1,\dots,S_{n-1},P)$ we are led to the following definitions.

\begin{defn}\label{def3}
    Let $(S_1, \dots, S_{n-1}, P)$ be a $\Gamma_n$-contraction on a Hilbert space $\mathcal
    H$. We say that $(S_1, \dots, S_{n-1},P)$ is
    \begin{itemize}
        \item[(i)] a c.n.u $\Gamma_n$-contraction if $P$ is a
        c.n.u contraction ;
        \item[(ii)] a c.n.i $\Gamma_n$-contraction if $P$ is a
        c.n.i contraction ;
        \item[(iii)] a \textit{weakly stable} $\Gamma_n$-contraction if $P$ is \textit{weakly
        stable}, that is, if $\langle P^nx,y \rangle \rightarrow 0$ as
        $n \rightarrow \infty$ for all $x, y \in \mathcal{H}$ ;
        \item[(iv)] a \textit{strongly stable} $\Gamma_n$-contraction if $P$ is \textit{strongly
        stable}, that is, if $\|P^nx\| \rightarrow 0$ as $n \rightarrow \infty$
        for all $x \in \mathcal{H}$ ;
        \item[(v)] a \textit{pure} $\Gamma_n$-contraction if
        $P^*$is strongly stable ;
        \item[(vi)] a $\mathcal{C}_{00}$ $\Gamma_n$-contraction if $P$ is
        $\mathcal{C}_{00}$-contraction, that is, both $P$ and $P^*$
        are strongly stable or in other word both $P$ and $P^*$
        are pure contractions ;
        \item[(vii)] a $\Gamma_n$-\textit{identity} if $P=I_{\mathcal
        H}$ and a \textit{completely non-identity}
        $\Gamma_n$-\textit{contraction} if there is no nontrivial
        proper subspace of $\mathcal H$ that reduces $P$ and on
        which $P$ is equal to the identity operator.
    \end{itemize}
\end{defn}

An analogue of Theorem \ref{thm:gamma-ui} holds for an $\mathbb
E$-contraction (see Theorems 5.4 and 5.7 in \cite{T:B}). Also
every $\mathbb E$-contraction admits a canonical decomposition
into an $\mathbb E$-unitary and a c.n.u $\mathbb E$-contraction
(Theorem 3.1 in \cite{S:P}). Therefore, it is legitimate to have
an analogue of Definition \ref{def3} for $\mathbb E$-contractions.

\begin{defn}\label{def3}
    Let $(A,B, P)$ be a $\mathbb E$-contraction on a Hilbert space $\mathcal
    H$. We say that $(A,B,P)$ is
    \begin{itemize}
        \item[(i)] a c.n.u $\mathbb E$-contraction if $P$ is a
        c.n.u contraction ;
        \item[(ii)] a c.n.i $\mathbb E$-contraction if $P$ is a
        c.n.i contraction ;
        \item[(iii)] a \textit{weakly stable} $\mathbb E$-contraction if $P$ is \textit{weakly
        stable}, that is, if $\langle P^nx,y \rangle \rightarrow 0$ as
        $n \rightarrow \infty$ for all $x, y \in \mathcal{H}$ ;
        \item[(iv)] a \textit{strongly stable} $\mathbb E$-contraction if $P$ is \textit{strongly
        stable}, that is, if $\|P^nx\| \rightarrow 0$ as $n \rightarrow \infty$
        for all $x \in \mathcal{H}$ ;
        \item[(v)] a \textit{pure} $\mathbb E$-contraction if
        $P^*$is strongly stable ;
        \item[(vi)] a $\mathcal{C}_{00}$ $\mathbb E$-contraction if $P$ is
        $\mathcal{C}_{00}$-contraction, that is, both $P$ and $P^*$
        are strongly stable or in other word both $P$ and $P^*$
        are pure contractions ;
        \item[(vii)] a $\mathbb E$-\textit{identity} if $P=I_{\mathcal
        H}$ and a \textit{completely non-identity}
        $\mathbb E$-\textit{contraction} if there is no nontrivial
        proper subspace of $\mathbb E$ that reduces $P$ and on
        which $P$ is equal to the identity operator.
    \end{itemize}
\end{defn}

\section{Preparatory results}\label{results}

We begin this section with a basic result from Chapter-3 of
\cite{V:P} which will be used frequently in the subsequent
sections.

\begin{lem}\label{VP1}
    Let $P,\, Q,\, A$ be operators on some Hilbert
    space $\mathcal{H}$ with $P$ and $Q$ being positive. Then
    $\begin{bmatrix}
    P & A\\
    A^* & Q
    \end{bmatrix}
    \geq 0$ if and only if $|\langle Ax, y \rangle|^2 \leq
    \langle Py, y \rangle \langle Qx, x \rangle$ for all $x, y$ in $\mathcal{H}$.
\end{lem}

A straight-forward corollary of Lemma \ref{VP1} is the following.

\begin{cor}\label{cor3.1}
    Let $P,\, Q,\, A$ be operators on some Hilbert space $\mathcal{H}$ with $P$ and $Q$ being positive. If
    $\begin{bmatrix}
    P & A\\
    A^* & Q
    \end{bmatrix}
    \geq 0$ and $P = 0$, then $A = 0$.
\end{cor}

\subsection{Results about $\Gamma_n$-contractions}

We start with a result from \cite{sourav3} that simplifies the
definition of a $\Gamma_n$-contraction by an application of
Oka-Weil theorem and the underlying reason is the fact that
$\Gamma_n$ is polynomially convex.

\begin{lem}[{\cite{sourav3}, Lemma 2.4}]\label{SP1}
    A commuting tuple of bounded operators
    $(S_1, \dots, S_{n-1}, P)$ is a $\Gamma_n$-contraction if and only if
    \[
    \|f(S_1, \dots, S_{n-1}, P)\| \leq \|f\|_{\infty, \Gamma_n}
    \]
    for any holomorphic polynomially $f$ in $n$-variables.
\end{lem}

The next theorem provides characterizations for a
$\Gamma_n$-unitary.
\begin{thm}[{\cite{S:B}, Theorem $4.2$}]\label{gammauni}
    Let $ S_1, \dots, S_{n-1} \text{ and }P $ be commuting
    operators on a Hilbert space $\mathcal{H}$. Then the following are equivalent:
    \begin{enumerate}
        \item $ (S_1, \dots, S_{n-1},P) $ is a $ \Gamma_n $-unitary;
        \item there exist commuting unitaries $U_1,\dots, U_n$ on
        $\mathcal H$ such that
        \[
         \pi_{n}(U_1,\dots,U_n)= (S_1, \dots, S_{n-1},P)\,;
        \]
        \item $ P $ is unitary, $ S_i = S_{n-i}^*P $ for each $\, i = 1, \dots, n-1 $ and
        $ \bigg(\dfrac{n-1}{n}S_1,\\ \dfrac{n-2}{n}S_2, \dots, \dfrac{1}{n}S_{n-1}\bigg) $
        is a $ \Gamma_{n-1} $-contraction.
    \end{enumerate}

\end{thm}

\noindent The following result is an analogue of Theorem
\ref{gammauni} for $\Gamma_n$-isometries.

\begin{thm}[{\cite{S:B}, Theorem 4.12}]\label{gammaiso}
    Let $ S_1, \dots, S_{n-1},P $ be commuting operators on a Hilbert space
    $\mathcal{H}$. Then the following are equivalent:
    \begin{enumerate}
        \item $ (S_1, \dots, S_{n-1},P) $ is a $ \Gamma_n $-isometry,
        \item $ P $ is isometry, $ S_i = S_{n-i}^*P $ for each $ i = 1, \dots, n-1 $ and
        $ \bigg(\dfrac{n-1}{n}S_1, \dfrac{n-2}{n}S_2,
        \dots, \dfrac{1}{n}S_{n-1}\bigg) $ is a $ \Gamma_{n-1} $-contraction.

    \end{enumerate}

\end{thm}

We imitate here a few sentences (until Proposition \ref{lem:3})
from \cite{sourav6} to recall $n-1$ operator pencils which will
determine the existence and uniqueness of the so called
$\ft$-tuple associated with a $\Gamma_n$-contraction. In
\cite{sourav3}, the second named author of this article introduced
the following $n-1$ operator pencils $\Phi_1,\dots,\Phi_{n-1}$ to
analyze the structure of a $\Gamma_n$-contraction
$(S_1,\dots,S_{n-1},P)$:
\begin{align}
\Phi_{i}(S_1,\dots, S_{n-1},P) &= (\tilde n_i-S_i)^*(\tilde
n_i-S_i)-(\tilde n_i P-S_{n-i})^*(\tilde n_i P-S_{n-i}) \notag
\\&
={\tilde n_i}^2(I-P^*P)+(S_i^*S_i-S_{n-i}^*S_{n-i})-\tilde n_i(S_i-S_{n-i}^*P) \notag \\
& \quad \quad -\tilde n_i(S_i^*-P^*S_{n-i})\,, \quad \quad \text{
where } \tilde n_i= \binom{n}{i} \label{eq:1a}.
\end{align}
We mention here to the readers that while defining $\Phi_i$ in
\cite{sourav3}, $\tilde n_i$ was mistakenly displayed as $n$ and
that was a typographical error. From the definition it is clear
that in particular when $S_1,\dots,S_{n-1}, P$ are scalars, i.e,
points in $\Gamma_n$, the above operator pencils take the
following form for each $i$ :
\begin{align}
\Phi_{i}(s_1,\dots,s_{n-1},p) & = {\tilde
n_i}^2(1-|p|^2)+(|s_i|^2-|s_{n-i}|^2)-\tilde
n_i(s_i-\bar{s}_{n-i}p) \notag \\ & \quad \quad -\tilde
n_i(\bar{s}_i-\bar{p}s_{n-i}). \label{eqn:2a}
\end{align}
The following result appeared in \cite{sourav3} and is important
in the context of this paper.

\begin{prop}[Proposition 2.6, \cite{sourav3}]\label{lem:3}
Let $(S_1,\dots,S_{n-1},P)$ be a $\Gamma_n$-contraction. Then for
$i=1,\dots,n-1,\; \Phi_i(\alpha
S_1,\dots,\alpha^{n-1}S_{n-1},\alpha^n P)\geq 0$ for all $\alpha
\in\overline{\mathbb D}$.
\end{prop}

The positivity of the operator pencils $\Phi_i$ determines a
unique operator tuple $(A_1,\dots, A_{n-1})$ associated with a
$\Gamma_n$-contraction $(S_1,\dots, S_{n-1},P)$. We name
$(A_1,\dots, A_{n-1})$ \textit{the fundamental operator tuple} or
in short the $\ft$-\textit{tuple} of $(S_1,\dots, S_{n-1},P)$ and
the underlying reason is that it plays the central role in every
section of operator theory on the symmetrized polyisc. We call
this result the existence-uniqueness theorem for $\ft$-tuples and
the result appeared in article \cite{sourav6} by the second named
author.

\begin{thm}[\cite{sourav6}, Theorem 3.3]\label{existence-uniqueness} Let
$(S_1,\dots,S_{n-1},P)$ be a $\Gamma_n$-contraction on a Hilbert
space $\mathcal H$. Then there are unique operators
$A_1,\dots,A_{n-1}\in\mathcal B(\mathcal D_P)$ such that
\[
S_i-S_{n-i}^*P=D_PA_iD_P \text{ for } i=1,\dots,n-1.
\]
Moreover, for each $i$ and for all $z\in \mathbb T$, $\omega
(A_i+A_{n-i}z)\leq \tilde n_i$. $[\tilde n_i= \binom{n}{i}]$.
\end{thm}

\subsection{Results about $\mathbb E$-contractions}

Here we recall from literature some results about $\mathbb
E$-contractions which are parallel to the corresponding results
about $\Gamma_n$-contractions described in the preceding
subsection. We begin with an analogue of Lemma \ref{SP1} for
$\mathbb E$-contractions.

\begin{lem}[{\cite{T:B}, Lemma 3.3}]\label{lemTB}
    A commuting triple of bounded operators $(A, B, P)$ is a \textit{tetrablock contraction} if and only if
    \[
    \|f(A, B, P)\| \leq \|f\|_{\infty, \mathbb E} = \text{sup}\{|f(z_1,z_2,z_3)| : (z_1,z_2,z_3)\in \mathbb E\}
    \]
    for any holomorphic polynomially $f$ in three variables.
\end{lem}

We now present analogues of Theorems \ref{gammauni} and
\ref{gammaiso} for $\mathbb E$-contractions.

\begin{thm}[{\cite{T:B}, Theorem 5.4}]\label{thm:tu}
    Let $\underline N = (N_1, N_2, N_3)$ be a commuting triple of
    bounded operators. Then the following are equivalent.

    \begin{enumerate}

        \item $\underline N$ is an $\mathbb E$-unitary,

        \item $N_3$ is a unitary and $\underline N$ is an $\mathbb
        E$-contraction,

        \item $N_3$ is a unitary, $N_2$ is a contraction and $N_1 = N_2^*
        N_3$.
    \end{enumerate}
\end{thm}

\begin{thm}[{\cite{T:B}, Theorem 5.7}] \label{thm:ti}

    Let $\underline V = (V_1, V_2, V_3)$ be a commuting triple of
    bounded operators. Then the following are equivalent.

    \begin{enumerate}

        \item $\underline V$ is an $\mathbb E$-isometry.

        \item $V_3$ is an isometry and $\underline V$ is an $\mathbb
        E$-contraction.

        \item $V_3$ is an isometry, $V_2$ is a contraction and $V_1=V_2^*
        V_3$.

    \end{enumerate}
\end{thm}

The next theorem will be used in the proof of almost every result
that we obtain in this paper about an $\mathbb E$-contraction.

\begin{thm}[\cite{T:B}, Theorem 3.5] \label{exist-tetra}
To every $\mathbb
E$-contraction $(A, B, P)$ there were two unique operators $F_1$
and $F_2$ on $\mathcal{\mathcal{D}_P} =\overline{\text{Ran}}(I -
P^*P)$ that satisfied the fundamental equations, i.e,
\[
A-B^*P = D_PF_1D_P\,, \qquad B-A^*P = D_PF_2D_P.
\]
\end{thm}
The operators $F_1,F_2$ are called the \textit{fundamental
operators} of $(A, B, P)$. Like the $\ft$-tuples of a
$\Gamma_n$-contraction, the fundamental operators play central
role in the theory of $\mathbb{E}$-contractions.

\section{Structure theorems for
$\Gamma_n$-contractions}\label{gamma-n}

We begin this section with the statement of the canonical decomposition of a $\Gamma_n$-contraction due to the second named author of this paper. It will help a reader to understand the further development of the structure theory (of a $\Gamma_n$-contraction) done in this article.

\begin{thm}[\cite{sourav3}, Theorem 1.1]\label{mainthm}
Let $(S_1,\dots,S_{n-1},P)$ be a $\Gamma_n$-contraction on a
Hilbert space $\mathcal H$. Let $\mathcal H_1$ be the maximal
subspace of $\mathcal H$ which reduces $P$ and on which $P$ is
unitary. Let $\mathcal H_2=\mathcal H\ominus \mathcal H_1$. Then
\begin{enumerate}
\item $\mathcal H_1,\mathcal H_2$ reduce $S_1,\dots, S_{n-1}$,
\item $(S_1|_{\mathcal H_1},\dots,S_{n-1}|_{\mathcal
H_1},P|_{\mathcal H_1})$ is a $\Gamma_n$-unitary, \item
$(S_1|_{\mathcal H_2},\dots,S_{n-1}|_{\mathcal H_2},P|_{\mathcal
H_2})$ is a completely non-unitary $\Gamma_n$-contraction.
\end{enumerate}
The subspaces $\mathcal H_1$ or $\mathcal H_2$ can be equal to the
trivial subspace $\{0\}$.
\end{thm}

As we have mentioned before that in \cite{N:L}, Levan showed that a c.n.u contraction could further
be decomposed orthogonally into a c.n.i contraction and a
unilateral shift (i.e., a pure contraction). We state the result
below.

\begin{thm}[\cite{N:L}, Theorem 1]\label{thmlv1}
    With respect to a c.n.u contraction $T$ and
    its adjoint $T^*$ on $\mathcal{H}$, $\mathcal{H}$ admits the unique orthogonal decomposition
    $$\mathcal{H} = \mathcal{H}_1\oplus\mathcal{H}_2,$$
    such that $T|_{\mathcal{H}_1}$ is a c.n.u isometry
    and $T|_{\mathcal{H}_2}$ is a c.n.i contraction.
\end{thm}

We present an analogue of the Theorem \ref{thmlv1} for
particular classes of completely non-unitary
$\Gamma_n$-contractions. This is one of the main results of this
article.

\begin{thm}\label{cd1}
    Let $(S_1, \dots, S_{n-1}, P)$ be a c.n.u $\Gamma_n$-contraction
    on a Hilbert space $\mathcal{H}$. Let $\mathcal{H}_1, \mathcal H_2$ be as in
    Theorem \ref{thmlv1}. If either $P^*$ commutes with each $S_i$ or $\mathcal{H}_1$
    is the maximal invariant subspace for $P$ on which $P$ is isometry, then
    \begin{enumerate}
        \item $\mathcal{H}_1$, $\mathcal{H}_2$ reduce $S_1, \dots, S_{n-1}$;
        \item $\left(S_1|_{\mathcal{H}_1}, \dots, S_{n-1}|_{\mathcal{H}_1},P|_{\mathcal{H}_1} \right)$
        is a c.n.u $\Gamma_n$-isometry;
        \item $\left(S_1|_{\mathcal{H}_2}, \dots, S_{n-1}|_{\mathcal{H}_2},P|_{\mathcal{H}_2} \right)$
        is a c.n.i $\Gamma_n$-contraction.
    \end{enumerate}
    Anyone of $\mathcal{H}_1,\mathcal{H}_2$ can be equal to the trivial subspace $\{0\}$.
\end{thm}
\begin{proof}
    \textbf{Case-I.} Let $P^*$ commute with $S_i$
    for $i = 1, \dots, n-1$. If $P$ is a c.n.i contraction
    then $\mathcal H_1=\{0\}$ and if $P$ is a c.n.u
    isometry then $\mathcal H=\mathcal H_1$ and so $\mathcal H_2=\{0\}$.
    In such cases the theorem is trivial. Suppose $P$ is neither a
    c.n.u isometry nor a c.n.i contraction. It is evident from Theorem \ref{thmlv1} that $\mathcal
    H_1 \subseteq \mathcal H$ is the maximal subspace which reduces
    $P$ and on which $P$ is a c.n.i isometry, i.e, a unilateral shift. Let us denote $P|_{\mathcal
    H_1}, P|_{\mathcal H_2}$ by $P_1$ and $P_2$ respectively.
    With respect to the decomposition $\mathcal H=\mathcal H_1\oplus
    \mathcal H_2$ suppose
    \[
    S_i=
    \begin{bmatrix}
    S_{i11}&S_{i12}\\
    S_{i21}&S_{i22}
    \end{bmatrix},
    \quad i=1,\dots, n-1.
    \]
    By Proposition \ref{lem:3}, we have for all $\alpha, \beta \in \mathbb{T}$,
    \begin{align*}
    \Phi_i(\alpha S_1, \dots, \alpha^{n-1}S_{n-1},\alpha^nP) = &
    \tilde{n}_i^2(I - P^*P) + (S_i^*S_i - S_{n-i}^*S_{n-i})
    \\& - 2\tilde{n}_i\text{ Re }\alpha^i(S_i - S_{n-i}^*P)\\ &  \geq 0,
    \end{align*}
    \begin{align*}
    \Phi_{n-i}(\beta S_1, \dots, \beta^{n-1}S_{n-1},\beta^nP) =
    & \tilde{n}_i^2(I - P^*P) + (S_{n-i}^*S_{n-i} - S_{i}^*S_{i})
    \\& - 2\tilde{n}_i\text{ Re }\beta^i(S_{n-i} - S_{i}^*P)\\ &  \geq 0,
    \end{align*}
    Adding $\Phi_i$ and $\Phi_{n-i}$ we get
    \[
    \tilde{n}_i(I-P^*P)-\text{Re }\alpha^i(S_i-S_{n-i}^*P)-\text{Re }
    \beta^{n-i}(S_{n-i}-S_i^*P)\geq 0
    \]
    that is
    \begin{align}\label{eqn:5}
    \begin{bmatrix}
    0&0\\
    0&\tilde{n}_i(I-P_2^*P_2)
    \end{bmatrix}
    -& \text{ Re }\alpha^i
    \begin{bmatrix}
    S_{i11}-S_{(n-i)11}^*P_1 & S_{i12}-S_{(n-i)21}^*P_2\\
    S_{i21}-S_{(n-i)12}^*P_1&S_{i22}-S_{(n-i)22}^*P_2
    \end{bmatrix}  \notag \\
    -&\text{ Re }\beta^{n-i}
    \begin{bmatrix}
    S_{(n-i)11}-S_{i11}^*P_1&S_{(n-i)12}-S_{i21}^*P_2\\
    S_{(n-i)21}-S_{i12}^*P_1&S_{(n-i)22}-S_{i22}^*P_2
    \end{bmatrix}\, \geq 0
    \end{align}
    for all $\alpha,\beta\in\mathbb T$. The matrix in the left
    hand side of (\ref{eqn:5}) is positive and hence self-adjoint. If we write (\ref{eqn:5}) as
    \begin{equation}\label{eqn:6}
    \begin{bmatrix}
    R&X\\
    X^*&Q
    \end{bmatrix}
    \geq 0\,,
    \end{equation}
    then
    \begin{align*}
    &(\mbox{i})\; R\,, Q \geq 0 \text{ and } R=-\text{ Re } \alpha^i (
    S_{i11}-S_{(n-i)11}^*P_1)\\& \qquad\qquad\qquad \qquad\quad\,\,\, -\text{ Re }\beta^{n-i}
    (S_{(n-i)11}-S_{i11}^*P_1)\\& (\mbox{ii}) X = -\frac{1}{2} \{
    \alpha^i (
    S_{i12}-S_{(n-i)21}^*P_2)+\bar{ \alpha}^i(S_{i21}^*-P_1^*S_{(n-i)12})\\& \qquad \quad \,\,  + \beta^{n-i}
    (S_{(n-i)12}-S_{i21}^*P_2)
    \text{  }\, +\bar{\beta}^{n-i}(S_{(n-i)21}^*-P_1^*S_{i12})
    \}
    \\&(\mbox{iii})\; Q=\tilde{n}_i (I-P_2^*P_2)-\text{ Re } \alpha^i (S_{i22}-S_{(n-i)22}^*P_2)\\& \qquad\qquad\qquad \qquad\qquad - \text{ Re }\beta^{n-i}
    (S_{(n-i)22}-S_{i22}^*P_2) \,.
    \end{align*}

    Since $R\geq 0$ for all $ \alpha$ and $\beta$, if we choose $\beta = \pm 1$, then for all $ \alpha\in\mathbb T$ we have
    \[
    \alpha^i(S_{i11}-S_{(n-i)11}^*P_1)+\bar{ \alpha}^i(S_{i11}^*-P_1^*S_{(n-i)11})\leq
    0.
    \]
    Now choosing $ \alpha =\pm 1$ we get
    \begin{equation}\label{eqn:7}
    (S_{i11}-S_{(n-i)11}^*P_1)+(S_{i11}^*-P_1^*S_{(n-i)11})=0
    \end{equation}
    and choosing $ \alpha =\pm i$ we get
    \begin{equation}\label{eqn:8}
    (S_{i11}-S_{(n-i)11}^*P_1)-(S_{i11}^*-P_1^*S_{(n-i)11})=0\,.
    \end{equation}
    Then, from (\ref{eqn:7}) and (\ref{eqn:8}) we get
    \[
    S_{i11} = S_{(n-i)11}^*P_1.
    \]
    Therefore, $R = 0$. Then by corollary \ref{cor3.1}, we have $X = 0$. Therefore,
    \begin{align*}
    &  \alpha^i (
    S_{i12}-S_{(n-i)21}^*P_2)+\bar{ \alpha}^i(S_{i21}^*-P_1^*S_{(n-i)12})+\beta^{n-i}
    (S_{(n-i)12}-S_{i21}^*P_2)\\&
    \qquad \qquad \qquad \qquad\qquad\qquad \qquad \qquad \quad +\bar{\beta}^{n-i}(S_{(n-i)21}^*-P_1^*S_{i12}) = 0,
    \end{align*}
    for all $ \alpha,\beta \in\mathbb T$. If we choose $\beta =\pm 1$, then for all $ \alpha \in \mathbb T$ we have
    \[
    \alpha^i (S_{i12}-S_{(n-i)21}^*P_2)+\bar{ \alpha}^i(S_{i21}^*-P_1^*S_{(n-i)12})=0\,.
    \]
    Choosing $ \alpha=1 \text{ and }i$, then we have
    \begin{equation}\label{eqn:9}
    S_{i12} = S_{(n-i)21}^*P_2 \quad\text{ and }\quad S_{i21}^* = P_1^*S_{(n-i)12}\,.
    \end{equation}

    Since $S_iP = PS_i$, so we have
    \begin{align}
    S_{i11}P_1&=P_1S_{i11}    & S_{i12}P_2=P_1S_{i12}\,, \label{eqn:2.5} \\
    S_{i21}P_1&=P_2S_{i21}    & S_{i22}P_2=P_2S_{i22}\,, \label{eqn:2.6}
    \end{align}
    From the first equation in (\ref{eqn:2.6}) we have that range of $S_{i21}$ is invariant under $P_2$.
    Again from $S_iP^*=P^*S_i$ we have
    \[
    S_{i21}P_1^* = P_2^*S_{i21},
    \]
    that is, the range of $S_{i21}$ is invariant under $P_2^*$ also. Therefore, range of $S_{i21}$
    is a reducing subspace for $P_2$. The first equation in (\ref{eqn:2.6})
    and the equations in (\ref{eqn:9}) provide
    \[
    P_2^*P_2S_{i21}  = P_2^*S_{i21}P_1 = S_{(n-i)12}^*P_1 = S_{i21}.
    \]
    Therefore, $P_2$ is isometry on the range of $S_{i21}$.
    But $P_2$ is a c.n.i contraction. So we must have $S_{i21}=0$. Again from (\ref{eqn:9}),
    we have $S_{i12} = 0$. Thus with respect to the decomposition
    $\mathcal{H} =\mathcal{H}_1 \oplus \mathcal{H}_2$
    \[
    S_i=
    \begin{bmatrix}
    S_{i11}&0\\
    0&S_{i22}
    \end{bmatrix}\, \text{ for all } i = 1, \dots, n-1.
    \]
    Thus both $\mathcal{H}_1$ and $\mathcal{H}_2$ reduce
    $S_1, \dots, S_{n-1}$. Now $(S_{111},\dots, S_{(n-1)11},P_1)$
    being the restriction of the $\Gamma_n$-contraction
    $(S_1, \dots, S_{n-1}, P)$ to the reducing subspace
    $\mathcal{H}_1$ is a $\Gamma_n$-contraction where $P_1$ is an isometry.
    Therefore, by theorem \ref{gammaiso}, $(S_{111},\dots, S_{(n-1)11},P_1)$
    on $\mathcal{H}_1$ is a c.n.u $\Gamma_n$-isometry.
    Since $P_2$ is a c.n.i contraction, $(S_{122}, \dots, S_{(n-1)22}, P_2)$
    on $\mathcal{H}_2$ is a c.n.i $\Gamma_n$-contraction.\\

    \textbf{Case-II.} Suppose $\mathcal{H}_1$ is the maximal invariant
    subspace for $P$ on which $P$ is isometry. Now from the first equation
    in (\ref{eqn:2.6}) and the equations in (\ref{eqn:9}), we have that the range of
    $S_{i21}$ is invariant under $P_2$ and
    \[
    P_2^*P_2S_{i21}  = P_2^*S_{i21}P_1 = S_{(n-i)12}^*P_1 = S_{i21}.
    \]
    This shows that $P_2$ is isometry on range of $S_{i21}$.
    Therefore, we must have $S_{i21}=0.$ Now from (\ref{eqn:9}),
    $S_{i12} = 0$. Using the same arguments as in the previous
    case we have that $(S_{111}, \dots, S_{(n-1)11}, P_1)$ on $\mathcal{H}_1$
    is a c.n.u $\Gamma_n$-isometry and $(S_{122},
    \dots, S_{(n-1)22}, P_2)$ on $\mathcal{H}_2$ is c.n.i $\Gamma_n$-contraction.
\end{proof}

The conclusions of the above theorem may fail if we drop either of
the conditions that $S_iP^*=P^*S_i$ for any $i = 1,\dots, n-1$ or
that $\mathcal{H}_1$ is the maximal invariant subspace for $P$ on
which $P$ is an isometry. We shall provide a counter example but
before that we recall a few necessary results. First we state the
well-known Ando dilation for a pair of commuting contractions
$(T_1,T_2)$.

\begin{thm}[Ando, \cite{Ando}]\label{thm2.1}
    For every commuting pair of contractions $(T_1, T_2)$ on a
    Hilbert space $\mathcal{H}$ there exists a commuting pair of
    unitaries $(U_1, U_2)$ on a Hilbert space $\mathcal{K}$
    containing $\mathcal{H}$ as a subspace such that
    \[
    T_1^{n_1}T_2^{n_2} = P_{\mathcal{H}}U_1^{n_1}U_2^{n_2}|_{\mathcal{H}}\,
    \qquad \text{ for } n_1, n_2 \geq 0.
    \]
\end{thm}

Since a commuting pair of contractions $(T_1,T_2)$ on $\mathcal H$
dilates to a commuting pair of unitaries $(U_1,U_2)$ defined on a
bigger Hilbert space $\mathcal K$, it is easy to verify that the
$n$-tuple of commuting contractions $(I_{\mathcal{H}}, \dots,
I_{\mathcal{H}}, T_1, T_2)$ on $\mathcal{H}$ can easily dilate to
the commuting $n$-tuple of unitaries $(I_{\mathcal{K}}, \dots,
I_{\mathcal{K}}, U_1, U_2)$ on $\mathcal{K}$. Therefore, von
Neumann inequality holds on the closed polydisc $\overline{\mathbb
D^n}$ for $(I_{\mathcal{H}}, \dots, I_{\mathcal{H}}, T_1, T_2)$,
that is,
    \[
    \|p(I_{\mathcal{H}}, \dots, I_{\mathcal{H}}, T_1, T_2)\| \leq \|p\|_{\infty, \bar{\mathbb{D}}^n}
    \,, \text{ for any } p\in\mathbb C[z_1,\dots,z_n]\,.
    \]
So, it follows that von Neumann inequality holds on
$\pi_n(\overline{\mathbb D^n})=\Gamma_n$ for $\pi_n
(I_{\mathcal{H}}, \dots, I_{\mathcal{H}}, T_1, T_2)$ which is same
as saying that $\pi_n (I_{\mathcal{H}}, \dots, I_{\mathcal{H}},
T_1, T_2)$ on $\mathcal H$ is a $\Gamma_n$-contraction.\\

\noindent Now let us consider the following example.

\begin{eg}\label{example1}
    Let $\mathcal{H} = \ell^2$, where
    \[
    \ell^2 = \left\{\{x_n\}_{n=1}^{\infty} : x_n \in \mathbb{C} \text{ and } \sum_{n=1}^{\infty}|x_n|^2 < \infty\right\}.
    \]
    Consider the operator $P : \ell^2 \rightarrow \ell^2$ defined by
    \[
    P(x_1, x_2, x_3, \dots) = (0, \frac{x_1}{2}, x_2, x_3, \dots).
    \]
    Now consider $T_1 = T_2 = P$. Then obviously
    $\pi_n(I_{\mathcal{H}}, \dots, I_{\mathcal{H}}, P, P)$ is a $\Gamma_n$-contraction. It is clear that
    \[
    P^{*2}P^2(x_1, x_2, x_3, \dots) = (\frac{x_1}{4}, x_2, x_3, \dots).
    \]
    Suppose $H = \{(0, x_1, x_2, \dots) : x_i \in \mathbb{C}\}$.
    Then clearly $H$ is the maximal invariant subspace for $P^2$
    on which $P^2$ is isometry. Suppose $H_1 (\subseteq H)$ is
    the maximal reducing subspace for $P^2$ on which $P^2$ is isometry.\\

   \noindent \textbf{Claim.} $H_1 = \{(0, x_1, 0, x_2, 0, x_3, \dots) : x_i \in \mathbb{C}\}$.\\

    \noindent \textbf{Proof of claim.}
    It is clear that $\{(0, x_1, 0, x_2, 0, x_3, \dots) : x_i \in \mathbb{C}\} \subseteq H_1$.
    Suppose $(0, \dots, 0, \underbrace{x_{2n+1}}_{\text{(2n+1)th position}}, 0, \dots) \in H_1$.
    Since $H_1$ is reducing subspace for $P^2$ so $P^{*2n}(0, \dots, 0, x_{2n+1}, 0, \dots) \in H_1$
    i.e., $\left( \dfrac{x_{2n+1}}{2}, 0, 0, \dots \right) \in H$. Therefore, we must have $x_{2n+1} = 0$.
    This completes the proof of the claim.\\
    One can easily check that $P^*P^2 \neq P^2P^*$
    and $ H \neq H_1$. It is clear that $H_1$ is not a
    reducing subspace for $P$. Therefore, the conclusion of Theorem \ref{cd1}
    is not true for the $\Gamma_n$-contraction
    $\pi_n(I_{\mathcal{H}}, \dots, I_{\mathcal{H}}, P, P)$ on $\mathcal{H} = \ell^2$.
\end{eg}

If $P$ is a contraction, then $\{ P^{*n}P^n: n \geq 1\}$ is a
non-increasing sequence of self-adjoint contractions so that it
converges strongly. Similarly, the sequence $\{P^nP^{*n}\ : n \geq
1 \}$ also converges strongly. Suppose $\mathcal{A}$ is the strong
limit of $\{ P^{*n}P^n : n \geq 1\}$ and $\mathcal{A}_*$ is the
strong limit of $\{P^nP^{*n}\ : n \geq 1 \}$. In \cite{K:C1}, the
following decomposition theorem for a contraction was proved by
Kubrusly.

\begin{thm}[{\cite{K:C1}, Theorem 1}]\label{Kdecom}
    Let $P$ be a contraction on $\mathcal{H}$. If $\mathcal{A} = \mathcal{A}^2$, then
    \[
    P = P_1 \oplus P_2 \oplus U,
    \]
    where $P_1$ is a strongly stable contraction acting on $\text{Ker}(\mathcal{A})$, $P_2$
    is a unilateral shift acting on $\text{Ker}(I-\mathcal{A})\cap \text{Ker}(\mathcal{A}_*)$
    and $U$ is a unitary acting on $\text{Ker}(I-\mathcal{A})\cap \text{Ker}(I-\mathcal{A}_*)$.
    Moreover, if $\mathcal{A} = \mathcal{A}^2$ and $\mathcal{A}_* = \mathcal{A}_*^2$, then
    \[
    P = P_1^{\circ} \oplus P_2^{\circ} \oplus P_2 \oplus U,
    \]
    where $P_1^{\circ}$ is a $\mathcal{C}_{00}$-contraction on
    $\text{Ker}(\mathcal{A})\cap \text{Ker}(\mathcal{A}_*)$ and $P_2^{\circ}$
    is a backward unilateral shift on $\text{Ker}(\mathcal{A})\cap \text{Ker}(I-\mathcal{A}_*).$
    Furthermore, if $\mathcal{A} = \mathcal{A}_*$, then
    \[
    P = P_1^{\circ} \oplus U.
    \]
\end{thm}

\noindent Here we present an analogue of Theorem \ref{Kdecom} for
a $\Gamma_n$-contraction and this is another main result of this
article.

\begin{thm}\label{mainthm}
    Let $(S_1, \dots, S_{n-1},P)$ be a $\Gamma_n$-contraction on a
    Hilbert space $\mathcal{H}$ and let $\mathcal A$ be as in Theorem \ref{Kdecom}.
    If $\mathcal{A} = \mathcal{A}^2$, then
    \begin{enumerate}
        \item $\text{Ker}(\mathcal{A})$, $\text{Ker}(I-\mathcal{A})\cap
        \text{Ker}(\mathcal{A}_*)$ and $\text{Ker}(I-\mathcal{A})\cap
        \text{Ker}(I-\mathcal{A}_*)$ reduce $S_1, \dots, S_{n-1}$,
        \item $\left(S_1|_{\text{Ker}(\mathcal{A})}, \dots, S_{(n-1)}
        |_{\text{Ker}(\mathcal{A})},P|_{\text{Ker}(\mathcal{A})}\right)$
        is a strongly stable $\Gamma_n$-contraction,
        \item $\left(S_1, \dots, S_{n-1}, P \right)$ on ${\text{Ker}
        (I-\mathcal{A})\cap \text{Ker}(\mathcal{A}_*)}$ is a pure $\Gamma_n$-isometry,
        \item $\left(S_1, \dots, S_{n-1}, P \right)$ on ${\text{Ker}
        (I-\mathcal{A})\cap \text{Ker}(I-\mathcal{A}_*)}$ is a $\Gamma_n$-unitary.
    \end{enumerate}
    Moreover, if $\mathcal{A} = \mathcal{A}^2$ and $\mathcal{A}_* = \mathcal{A}_{*}^2$, then
    \begin{enumerate}
        \item[$(1')$] $\left(S_1|_{\text{Ker}(\mathcal{A})\cap
        \text{Ker}(\mathcal{A}_{*})}, \dots, S_{(n-1)}|_{\text{Ker}
        (\mathcal{A})\cap \text{Ker}(\mathcal{A}_{*})},P|_{\text{Ker}
        (\mathcal{A})\cap \text{Ker}(\mathcal{A}_{*})}\right)$ is
        a $\mathcal{C}_{00}$  $\Gamma_n$-contraction,
        \item[$(2')$] $\left(S_1, \dots, S_{n-1}, P \right)$ on
        ${\text{Ker}(\mathcal{A})\cap \text{Ker}(I-\mathcal{A}_{*})}$
        is a $\Gamma_n$-contraction such that $P$ is a backward unilateral shift.
    \end{enumerate}
    Furthermore, if $\mathcal{A} = \mathcal{A}_{*}$, then
    \begin{enumerate}
        \item[$(1'')$] $\big(S_1|_{\text{Ker}(\mathcal{A})}, \dots, S_{(n-1)}
        |_{\text{Ker}(\mathcal{A})},P|_{\text{Ker}(\mathcal{A})}\big)$
        is a $\mathcal{C}_{00}$  $\Gamma_n$-contraction,
        \item[$(2'')$] $\big(S_1|_{\text{Ker}(I-\mathcal{A})},
        \dots, S_{(n-1)}|_{\text{Ker}(I-\mathcal{A})},
        P|_{\text{Ker}(I-\mathcal{A})}\big)$ is a $\Gamma_n$-unitary.
    \end{enumerate}
\end{thm}
\begin{proof}
    If $ \mathcal{A} = \mathcal{A}^2 $, then by
    Proposition 3.3 in \cite{K:C} we have $\mathcal{H} = \text{Ker}
    (I-\mathcal{A}) \oplus \text{Ker}(\mathcal{A}).$ Since $\text{Ker}
    (\mathcal{A})$ and $ \text{Ker}(I - \mathcal{A})$ are invariant
    subspaces for $P$, by Proposition 3.1 in \cite{K:C}, they reduce $P$.
    Thus we have the following decomposition
    \[
    P = V \oplus P_1 \; \text{ on } \; \mathcal{H} = \text{Ker}
    (I-\mathcal{A}) \oplus \text{Ker}(\mathcal{A})\, ,
    \]
    where $P_1 = P|_{\text{Ker}(\mathcal{A})}$ is a strongly stable
    contraction and $V = P|_{\text{Ker}
    (I-\mathcal{A})}$ is an isometry.
    With respect to the decomposition $\mathcal{H} =
    \text{Ker}(I - \mathcal{A}) \oplus \text{Ker}(\mathcal{A}) $, let
    \[
    P =
    \begin{bmatrix}
        V&0\\
        0&P_1
    \end{bmatrix}
    \text{ and }
    S_i =
    \begin{bmatrix}
    S_{i11} & S_{i12}\\
    S_{i21} & S_{i22}
    \end{bmatrix}
    \text{ for all } i = 1, \dots, n-1.
    \]
    Let $(F_1, \dots,F_{n-1})$ be the $\ft$-tuple of $(S_1,\dots, S_{n-1}, P)$.
    Then,
    \[
    S_i - S_{n-i}^*P = D_PF_iD_P, \text{ for } i = 1, \dots, n-1.
    \]
    With respect to the decomposition $\mathcal{D}_P = \{0\}\oplus\mathcal{D}_{P_1}$, let
    \[ F_i=
    \begin{bmatrix}
    0 & 0\\
    0 & F_{i22}
    \end{bmatrix}
    \text{ for } i = 1, \dots, n-1.
    \]
    Then from $S_i - S_{n-i}^*P = D_PF_iD_P$, we have
    \begin{align}
    S_{i11}&=S_{(n-i)11}^*V      & S_{i12}&=S_{(n-i)21}^*P_1\,, \label{eqn:3.1} \\
    S_{i21}&=S_{(n-i)12}^*V    & S_{i22} - S_{(n-i)22}^*P_2&=D_{P_1}F_{122}D_{P_1}\,. \label{eqn:3.2}
    \end{align}
    Since $S_iP = PS_i$ for all $i = 1, \dots, n-1$, so we have
    \begin{align}
    S_{i11}V&=VS_{i11}    & S_{i12}P_1=VS_{i12}\,, \label{eqn:3.3} \\
    S_{i21}V&=P_1S_{i21}    & S_{i22}P_1=P_1S_{i22}\,. \label{eqn:3.4}
    \end{align}
    Now from the first equation in (\ref{eqn:3.4}),
    we have that the range of $S_{i21}$ is an invariant subspace for $P_1$
    and by equation (\ref{eqn:3.2})
    \[
    P_1^*P_1S_{i21} =  P_1^*S_{i21}V
    =(S_{i21}^*P_1)^*V
    =S_{(n-i)12}^*V
    =S_{i21}.
    \]
    This implies that $P_1$ is an isometry on the range of $S_{i21}$.
    Since $P_1$ is strongly stable contraction on $\text{Ker}
    (\mathcal{A})$, so we must have $S_{i21} = 0$. Then from
    the second equation in (\ref{eqn:3.1}), we have
    $ S_{(n-i)12} = 0 $. Therefore, with respect to
    the decomposition $\mathcal{H} =
    \text{Ker}(I - \mathcal{A}) \oplus \text{Ker}(\mathcal{A})$
    \[
    S_i
    =
    \begin{bmatrix}
    S_{i11} & 0\\
    0 & S_{i22}
    \end{bmatrix}  \text{ for all } i = 1, \dots, n-1.
    \]
    So $\text{Ker}(\mathcal{A})$ and $\text{Ker}(I-\mathcal{A})$ reduce
    $S_1, \dots, S_{n-1}$.
    Therefore, the two tuples $(S_{111}, \dots, S_{(n-1)11}, V)$ and $(S_{122},
    \dots, S_{(n-1)22}, P_1)$, by being the restrictions of the
    $\Gamma_n$-contraction $(S_1,\dots, S_{n-1},P)$
    to a joint reducing subspaces, are $\Gamma_n$-contractions on
    $\text{Ker}(I-\mathcal{A})$ and $\text{Ker}(\mathcal{A})$
    respectively. Since $V$ is an isometry on
    $\text{Ker}(I-\mathcal{A})$,
    it follows from Theorem \ref{gammaiso} that $(S_{111}, \dots, S_{(n-1)11}, V)$ is
    a $\Gamma_n$-isometry. Again $P_1$
    is a strongly stable contraction on $\text{Ker}(\mathcal{A})$,
    so $(S_{122}, \dots, S_{(n-1)22}, P_1)$ is a strongly
    stable $\Gamma_n$-contraction on $\text{Ker}(\mathcal{A})$.\\

    Since $V$ is an isometry on $\text{Ker}(I-\mathcal{A})$, so by Wold
    decomposition (see \cite{Nagy}, Section-I) $\text{Ker}(I-\mathcal{A})$
    decomposes into an orthogonal direct sum $\text{Ker}(I-\mathcal{A}) =
    \mathcal{U} \oplus \mathcal{U}^{\perp}$ such that $\mathcal{U}$
    and $\mathcal{U}^{\perp}$ reduce $V$, $V|_{\mathcal{U}}$ is a unitary and
    $V|_{\mathcal{U}^{\perp}}$ is
    a unilateral shift. Following the proof of Theorem 5.8 in \cite{K:C}, we have
    \[
    \mathcal{U} = \text{Ker}(I-\mathcal{A})\cap \text{Ker}(I-\mathcal{A}_{*})
    \text{ and } \mathcal{U}^{\perp} = \text{Ker}(I-\mathcal{A})\cap \text{Ker}(\mathcal{A}_{*}).
    \]
    Since $(S_{111}, \dots, S_{(n-1)11}, V)$ is a $\Gamma_n$-isometry on
    $\text{Ker}(I - \mathcal{A})$ so by Theorem 4.12 in \cite{S:B}, we have
    \begin{enumerate}
        \item[(i)] $\mathcal{U} \text{ and }\mathcal{U}^{\perp}$
        reduce $S_{111}, \dots, S_{(n-1)11} $\,,
        \item[(ii)] $(S_{111}, \dots, S_{(n-1)11}, V)$ on
        $\mathcal{U}^{\perp}$ is a pure $\Gamma_n$-isometry\,,
        \item[(iii)] $(S_{111}, \dots, S_{(n-1)11}, V)$
        on $\mathcal{U}$ is a $\Gamma_n$-unitary.
    \end{enumerate}
Let $\mathcal{A}_{*}^{\prime}$ be the strong operator limit of
$\{P^nP^{*n} : n\geq 1\}$ on $\text{Ker}(\mathcal{A})$. Suppose
$\mathcal{A} = \mathcal{A}^2$ and $\mathcal{A}_* =
\mathcal{A}_{*}^2$. Then from Theorem 5.8 in \cite{K:C}, we have
the following decomposition
\[
P_1 = P_1^{\circ} \oplus P_2^{\circ}
\]
on $\text{Ker}(\mathcal{A}) = \text{Ker}(\mathcal{A}_{*}^{\prime})
\oplus \text{Ker}(I-\mathcal{A}_{*}^{\prime})$, where
$P_1^{\circ}$ is a $\mathcal{C}_{00}$-contraction on
$\text{Ker}(\mathcal{A})\cap \text{Ker}(\mathcal{A}_{*})$ and
$P_2^{\circ}$ is a backward unilateral shift on
$\text{Ker}(\mathcal{A})\cap \text{Ker}(I-\mathcal{A}_{*}).$ Using
the similar arguments as in the previous case for the
$\Gamma_n$-contraction $(S_1^*|_{\text{Ker}(\mathcal{A})}, \dots,
S_{(n-1)}^*|_{\text{Ker}(\mathcal{A})},
P^*|_{\text{Ker}(\mathcal{A})})$ on $\text{Ker}(\mathcal{A})$ we
obtain the following decomposition
    \begin{enumerate}
        \item[$(1')$] $\left(S_1|_{\text{Ker}(\mathcal{A})\cap \text{Ker}(\mathcal{A}_{*})},
        \dots, S_{(n-1)}|_{\text{Ker}(\mathcal{A})\cap \text{Ker}(\mathcal{A}_{*})},
        P|_{\text{Ker}(\mathcal{A})\cap \text{Ker}(\mathcal{A}_{*})}\right)$
        is a  $\mathcal{C}_{00}$  $\Gamma_n$-contraction;
        \item[$(2')$] $\left(S_1^*, \dots, S_{(n-1)}^*,P^* \right)$ is a completely
        non-unitary $\Gamma_n$-isom-\\etry on $ \text{Ker}(\mathcal{A})\cap \text{Ker}(I-\mathcal{A}_{*})$.
    \end{enumerate}
    Furthermore, if $\mathcal{A}=\mathcal{A}_{*}$, then by Proposition 3.4
    in
    \cite{K:C}, we have that $\mathcal{A} = \mathcal{A}^2.$ So $\text{Ker}
    (I-\mathcal{A}_{*}) \cap \text{Ker}(\mathcal{A}) = \{0\}$ and $\text{Ker}
    (I-\mathcal{A}) \cap \text{Ker}(\mathcal{A}_{*}) = \{0\}$. Therefore if $\mathcal{A} = \mathcal{A}_{*}$, we have
    $$\mathcal{H} = \text{Ker}(\mathcal{A}) \oplus \text{Ker}(I-\mathcal{A}).$$
    Thus $\big(S_1|_{\text{Ker}(\mathcal{A})}, \dots, S_{(n-1)}|_{\text{Ker}
    (\mathcal{A})},P|_{\text{Ker}(\mathcal{A})}\big)$ is a $\mathcal{C}_{00} \;
    \Gamma_n$-contraction on $\text{Ker}(\mathcal{A})$ and $\big(S_1|_{\text{Ker}(I-\mathcal{A})},
    \dots, S_{(n-1)}|_{\text{Ker}(I-\mathcal{A})},P|_{\text{Ker}(I-\mathcal{A})}\big)$ is a
    $\Gamma_n$-unitary on $\text{Ker}(I-\mathcal{A})$.

\end{proof}

We recall from the literature (see \cite{L:S}) that a contraction
$P$ on a Hilbert space $\mathcal{H}$ induces the orthogonal
decomposition
\[
\mathcal{H} = \text{Ker}(I-P) \oplus \overline{\text{Ran}}(I-P)
\]
where $\text{Ker}(I-P)$ and $\overline{\text{Ran}}(I-P)$ are
reducing subspaces for $P$. Here we produce an analogous
decomposition for a $\Gamma_n$-contraction.
\begin{thm}\label{thm:11}
    Let $\left( S_1, \dots, S_{n-1}, P \right)$ be a $\Gamma_n$-contraction
    on a Hilbert space $\mathcal{H}$. Then
    \begin{enumerate}
        \item $ \overline {\text{Ran}}(I-P)$, $\text{Ker}(I-P)$ reduce $ S_1, \dots, S_{n-1} $;
        \item  $\left(S_1, \dots, S_{n-1}, P \right)$ on $\text{Ker}(I-P)$ is a
        $\Gamma_n$-identity ;
        \item  $\left( S_1, \dots, S_{n-1}, P \right)$
        on $\overline{\text{Ran}}(I-P)$ is a completely non-identity $
        \Gamma_n$-contraction.
    \end{enumerate}
\end{thm}
\begin{proof}
It is evident that the restriction of a $\Gamma_n$-contraction
$(S_1,\dots, S_{n-1},P)$ to a joint reducing or invariant subspace
of $S_1,\dots, S_{n-1},P$ is also a $\Gamma_n$-contraction.
Therefore, it suffices to prove that $\text{Ker}(I-P)$ is a
reducing subspace for each $S_i$ because then $\overline{
\text{Ran}}(I-P)$ also reduces each $S_i$. One can easily check
that $\text{Ker}(I-P) =
    \text{Ker}(I-P^*)$. Now for any $x \in \text{Ker}(I-P)$ we have
    \[
    (I-P)S_ix = S_i(I-P)x = 0, \text{  for }i=1,\dots,n-1.
    \]
    Again for any $x \in \text{Ker}(I-P) = \text{Ker}(I-P^*)$ we have
    \[
    (I-P^*)S_i^*x = S_i^*(I - P^*)x = 0, \text{  for all }i.
    \]
    Thus $\text{Ker}(I-P)$ is a reducing subspace for each $S_i$. Since $P$ is identity on
    $\text{Ker}(I-P)$ and completely non-identity on
    $\overline{\text{Ran}}(I-P)$, it follows that
    $(S_1,\dots,S_{n-1},P)$ on $\text{Ker}(I-P)$ is
    $\Gamma_n$-identity and on $\overline{\text{Ran}}(I-P)$ is a completely
    non-identity $\Gamma_n$-contraction.
\end{proof}

The following theorem provides a decomposition of a c.n.u
contraction $P$ on a Hilbert space $\mathcal{H}$ but in this case
the orthogonal subspaces may not be reducing subspaces for $P$.
\begin{thm}[{\cite{N:L}, Theorem 3}]\label{NL1}
    For a completely non-unitary contraction $P$ and its adjoint $P^*$
    on $\mathcal{H}$, the following orthogonal decomposition holds,
    \begin{equation}\label{9}
    \mathcal{H} = \mathcal{H}_1^{\prime} \oplus \mathcal{H}_2^{\prime},
    \end{equation}
    where
    \begin{enumerate}
        \item $\mathcal{H}_1^{\prime}$ is maximal subspace of $\mathcal{H}$
        such that $P|_{\mathcal{H}_1^{\prime}}$ is c.n.u isometry;
        \item $P^*|_{\mathcal{H}_2^{\prime}}$ is a c.n.u contraction.
    \end{enumerate}
    Moreover,
    \[
    \lim_{n\rightarrow\infty}\|P^nx\| < \|x\|,  \text{ for all } x\in \mathcal{H}_2^{\prime}.
    \]
    In this case $\mathcal{H}_1^{\prime}$ is invariant subspace for $P$;
    $\mathcal{H}_2^{\prime}$ is invariant subspace for $P^*$.
    The decomposition $\eqref{9}$ is unique.
\end{thm}

\noindent We find an immediate corollary of the previous theorem.

\begin{cor}\label{cor1}
    Let $P$ be a completely non-unitary contraction on a Hilbert space
    $\mathcal{H}$. Let the orthogonal decomposition $\mathcal{H} =
    \mathcal{H}_1^{\prime} \oplus \mathcal{H}_2^{\prime}$ be as in
    Theorem \ref{NL1}. Then $\mathcal{H}_1^{\prime}$ is a reducing subspace for $P^*P$.
\end{cor}
\begin{proof}
    With respect to the decomposition $\mathcal H=\mathcal H_1^{\prime}\oplus
    \mathcal H_2^{\prime}$, let
    \[
    P=
    \begin{bmatrix}
    P_{11}&P_{12}\\
    0&P_{22}
    \end{bmatrix}
    \]
    so that $P_{11}$ is a completely non-unitary isometry
    on $\mathcal{H}_1^{\prime}$. It is known from \cite{N:L} that
    \[
    \mathcal{H}_1' = \left\{ h\in \mathcal{H}:\|P^nh\|=\|h\| \text{ for }n\geq 1 \right\}.
    \]
    Then for any $h \in \mathcal{H}_1'$\,,
    \begin{align*}
        \|P_{11}^*P_{11}h - h\|^2 & = \langle P_{11}^*P_{11}h - h, P_{11}^*P_{11}h - h \rangle \\
        & = \langle P_{11}^*P_{11}h, P_{11}^*P_{11}h \rangle -
        \langle P_{11}^*P_{11}h, h \rangle - \langle h, P_{11}^*P_{11}h \rangle + \|h\|^2 \\
        & = \langle P_{11}^*P_{11}h, P_{11}^*P_{11}h \rangle -
        \langle P_{11}h, P_{11}h \rangle - \langle P_{11}h, P_{11}h \rangle + \|h\|^2 \\
        & = \langle P_{11}^*P_{11}h, P_{11}^*P_{11}h \rangle
        - \langle Ph, Ph \rangle - \langle Ph, Ph \rangle + \|h\|^2 \\
        & = \|P_{11}^*P_{11}h\|^2 - \|h\|^2 \leq 0\,.
    \end{align*}
    Therefore, $P_{11}^*P_{11} = I_{\mathcal{H}_1'}$.
    Since $P$ is a contraction so $I - P^*P \geq 0$. Now
    \[
    I - P^*P =
    \begin{bmatrix}
    0 & -P_{11}^*P_{12} \\
    -P_{12}^*P_{11} & I - P_{12}^*P_{12} - P_{22}^*P_{22}
    \end{bmatrix}
    \geq 0\,.
    \]
    Then by Corollary \ref{cor3.1} we have $P_{11}^*P_{12} = 0$ and hence
    \[
    P^*P =
    \begin{bmatrix}
    P_{11}^*P_{11} & 0 \\
    0 & P_{12}^*P_{12} + P_{22}^*P_{22}
    \end{bmatrix}.
    \]
    Therefore, $\mathcal{H}_1'$ is a reducing subspace for $P^*P$.
\end{proof}

We now present here an analogue of the Theorem \ref{NL1} for a
particular class of c.n.u $\Gamma_n$-contractions.
\begin{thm}\label{cd4}
    Let $(S_1, \dots, S_{n-1},P)$ be a c.n.u $\Gamma_n$-contraction
    on $\mathcal{H}$ such that $P^*$ commutes with each $S_i$ and let
    $\mathcal{H}_1^{\prime}\,,\,\mathcal{H}_2^{\prime}$ be as in Theorem \ref{NL1}.
    Then with respect to the decomposition
    \[
    \mathcal{H} = \mathcal{H}_1^{\prime}\oplus\mathcal{H}_2^{\prime},
    \]
    the following are true:
    \begin{enumerate}
        \item $\mathcal{H}_1^{\prime}$ is a maximal joint invariant
        subspace for $S_1, \dots, S_{n-1} \text{ and } P$ such that
        $\left(S_1|_{\mathcal{H}_1^{\prime}},\dots, S_{n-1}
        |_{\mathcal{H}_1^{\prime}}, P|_{\mathcal{H}_1^{\prime}}\right)$
        is c.n.u $\Gamma_n$-isometry;
        \item $\mathcal{H}_2^{\prime}$ is a maximal joint
        invariant subspace for $S_1^*, \dots, S_{n-1}^* \text{ and }
        P^*$ such that $\left(S_1^*, \dots, S_{n-1}^*,P^* \right)$
        on ${\mathcal{H}_2^{\prime}}$ is c.n.u $\Gamma_n$-contraction.
    \end{enumerate}
    The above decomposition is unique.
\end{thm}
\begin{proof}
    By virtue of the Theorem \ref{NL1} we know that,
    \[
    \mathcal{H}_1^{\prime} = \left\{h\in \mathcal{H} : \|P^nh\| = \|h\| \text{ for } n\geq 1\right\}.
    \]
    Then for any $h \in \mathcal{H}_1'$\,,
    \begin{align*}
    \|P^{*n}P^nh - h\|^2 & = \langle P^{*n}P^nh - h, P^{*n}P^nh - h \rangle \\
    & = \langle P^{*n}P^nh, P^{*n}P^nh \rangle -
    \langle P^{*n}P^nh, h \rangle - \langle h, P^{*n}P^nh \rangle + \|h\|^2 \\
    & = \|P^{*n}P^nh\|^2 - \|h\|^2 \leq 0\,.
    \end{align*}
    This implies that $P^{*n}P^nh = h$ for any $n$.
    Since for any $i= 1,\dots, n-1$, $S_i^*P = PS_i^*$
    so it is clear that for any $h\in \mathcal{H}_1^{\prime}$ we have
    \[
        \|P^nS_ih\|^2  = \langle P^{*n}P^nS_ih, S_ih \rangle
        = \langle S_iP^{*n}P^nh, S_ih \rangle = \langle S_ih, S_ih \rangle
        = \|S_ih\|^2.
    \]
    Therefore, $\mathcal{H}_1^{\prime}$ is invariant subspace for $S_i$ and
    similarly $\mathcal{H}_2^{\prime}$ is invariant subspace for $S_i^*$ for any
    $i = 1, \dots, n-1$. Therefore, both $\left(S_1|_{\mathcal{H}_1^{\prime}},
    \dots, S_{n-1}|_{\mathcal{H}_1^{\prime}},
    P|_{\mathcal{H}_1^{\prime}}\right)$ and $\big(S_1^*|_{\mathcal{H}_2^{\prime}},\dots, S_{n-1}^*|_{\mathcal{H}_2^{\prime}},
    P^*|_{\mathcal{H}_2^{\prime}}\big)$ are
    $\Gamma_n$-contractions of which $P|_{\mathcal{H}_1^{\prime}}$
    is a c.n.u isometry and $P^*|_{\mathcal{H}_2^{\prime}}$ is a
    c.n.u contraction. Hence $\left(S_1|_{\mathcal{H}_1^{\prime}},
    \dots, S_{n-1}|_{\mathcal{H}_1^{\prime}},
    P|_{\mathcal{H}_1^{\prime}}\right)$ is a c.n.u
    $\Gamma_n$-isometry and $\big(S_1^*|_{\mathcal{H}_2^{\prime}},\dots, S_{n-1}^*|_{\mathcal{H}_2^{\prime}},
    P^*|_{\mathcal{H}_2^{\prime}}\big)$ is a c.n.u
    $\Gamma_n$-contraction. Obviously this decomposition is
    unique.

\end{proof}

\noindent \textbf{Note 1.} If $\mathcal{H}_1^{\prime} \neq \{0\}$,
then dimension of $\mathcal{H}_1^{\prime}$ can not be finite.
Otherwise, $\mathcal{H}_1^{\prime}$ will be a reducing subspace
for $P$ on which $P$ is unitary.\\

\noindent \textbf{Note 2.} Like Theorem \ref{mainthm}, if we drop the condition that $P^*$ commutes with each $S_i$ from hypothesis, we may not achieve the desired decomposition. The same example (Example \ref{example1}) can be referred to as a counter example here and it can be easily verified.\\

A direct consequence of the Theorem \ref{NL1} is the following
theorem of one-variable operator theory.
\begin{thm}\label{NL2}
    Let $P$ be a completely non-unitary contraction operator on
    $\mathcal{H}$, then $\mathcal{H}$ admits the unique orthogonal decomposition
    \[
    \mathcal{H} = \mathcal{H}_1^{\circ} \oplus \mathcal{H}_2^{\circ} \oplus \mathcal{H}_3^{\circ},
    \]
    where $\mathcal{H}_1^{\circ}$ and $\mathcal{H}_3^{\circ}$
    are invariant for $P$ while $\mathcal{H}_2^{\circ}$ is invariant for $P^*$. Moreover,
    \begingroup
    \allowdisplaybreaks
    \begin{align*}
    &\lim\limits_{n\rightarrow\infty}\|P^nx\| = 0,  \text{ for each }  x\in \mathcal{H}_1^{\circ},\\
    &\lim\limits_{n\rightarrow\infty}\|P^nx\| < \|x\|,  \text{ for each }  x\in \mathcal{H}_2^{\circ}
    \end{align*}
    \endgroup
    and
    \[
    \lim\limits_{n\rightarrow\infty}\|P^nx\| = \|x\|,  \text{ for each }  x\in \mathcal{H}_3^{\circ}.
    \]
\end{thm}

\noindent The following theorem provides an analogue of Theorem
\ref{NL2} for $\Gamma_n$-contractions.

\begin{thm}\label{cd5}
    Let $(S_1,\dots, S_{n-1}, P)$ be a c.n.u
    $\Gamma_n$-contraction on $\mathcal{H}$ such that $P^*$
    commutes with each $S_I$. Let $\mathcal{H}_1^{\circ}\,,\, \mathcal{H}_2^{\circ} \,,\,
    \mathcal{H}_3^{\circ}$ be as in Theorem \ref{NL2}. Then with
    respect to the orthogonal decomposition
    \[
    \mathcal{H} = \mathcal{H}_1^{\circ} \oplus \mathcal{H}_2^{\circ} \oplus \mathcal{H}_3^{\circ},
    \]
    the following statements hold:
    \begin{enumerate}
        \item $\mathcal{H}_1^{\circ}$ and $\mathcal{H}_3^{\circ}$ are
        joint invariant subspaces for $S_1,\dots, S_{n-1} \text{ and } P$;
        \item $\mathcal{H}_2^{\circ}$ is a joint invariant subspace for
        $S_1^*,\dots, S_{n-1}^* \text{ and } P^*$\,;
        \item $\left( S_1, \dots, S_{n-1}, P \right)$ on
        ${\mathcal{H}_1^{\circ}}$ is a strongly stable $\Gamma_n$-contraction;
        \item $\left( S_1^*|_{\mathcal{H}_2^{\circ}}, \dots, S_{n-1}^*
        |_{\mathcal{H}_2^{\circ}}, P^*|_{\mathcal{H}_2^{\circ}} \right)$
        is a c.n.u $\Gamma_n$-contraction;
        \item $\left( S_1|_{\mathcal{H}_3^{\circ}}, \dots, S_{n-1}
        |_{\mathcal{H}_3^{\circ}}, P|_{\mathcal{H}_3^{\circ}} \right)$ is a c.n.u $\Gamma_n$-isometry.
    \end{enumerate}
\end{thm}
\begin{proof}
    The subspace $\mathcal{H}_3^{\circ}$ is actually $\mathcal{H}_1^{\prime}$,
    while $\mathcal{H}_1^{\circ}$ is taken to be the subspace
    \[
    M(P) = \{x\in \mathcal{H} :  \|P^nx\| \rightarrow 0, n\rightarrow\infty\}.
    \]
    Now it is clear that $M(P)\subset \mathcal{H}_2^{\prime}$. Then
    $\mathcal{H}_2^{\circ}$ must be the orthogonal complement of
    $M(P)$ in $\mathcal{H}_2^{\prime}$. Suppose $x \in M(P)$. Then
    \[
    \|P^nS_ix\| =\|S_iP^nx \|\leq \|S_i\|\|P^nx\| \rightarrow 0, \text{ as } n \rightarrow \infty \,.
    \]
    Therefore, $M(P)$ is a joint invariant subspace for $S_1,\dots,
    S_{n-1} \text{ and } P$ and consequently, $\mathcal{H}_2^{\circ}$
    is a joint invariant subspace for $S_1^*,\dots, S_{n-1}^*$ and
    $P^*$. Then by using Theorem \ref{SP1} we have that $(S_1, \dots, S_{n-1},P)$
    on $\mathcal{H}_3^{\circ}$, $(S_1, \dots, S_{n-1}, P)$ on
    $\mathcal{H}_1^{\circ}$ and $(S_1^*, \dots, S_{n-1}^*,P^*)$
    on $\mathcal{H}_2^{\circ}$ are $\Gamma_n$-contractions. Since
    $P$ on $\mathcal{H}_1^{\circ}$ is strongly stable so
    $(S_1, \dots, S_{n-1}, P)$ on $\mathcal{H}_1^{\circ}$
    is a strongly stable $\Gamma_n$-contraction. Again $P$
    on $\mathcal{H}_3^{\circ}$ is an isometry. Thus
    $(S_1|_{\mathcal{H}_3^{\circ}}, \dots, S_{n-1}|_{\mathcal{H}_3^{\circ}},
    P|_{\mathcal{H}_3^{\circ}})$ is a c.n.u $\Gamma_n$-isometry.
\end{proof}

\noindent \textbf{Note.} The condition that $P^*$ commutes with
$S_1,\dots,S_{n-1}$ in the hypothesis cannot be ignored. An explanation is given in Note 2 after Theorem \ref{cd4}.\\

The following decomposition of a contraction was found by Foguel.
Since a bilateral shift is a weakly stable unitary operator, the
decomposition is not unique.
\begin{thm}[{\cite{Foguel}, Theorem 1.1}]\label{Foguel}
    Let $P$ be a contraction on a Hilbert space $\mathcal{H}$ and set
    $$\mathcal{E} = \left\{x \in \mathcal{H}:\langle P^nx, y \rangle
    \rightarrow 0 \text{ as } n \rightarrow \infty, \text{ for all } y \in \mathcal{H} \right\}.$$
    Then $\mathcal{E}$ is a reducing subspace for $P$. Moreover, the decomposition
    $$P = Z \oplus U$$
    on $\mathcal{H} = \mathcal{E} \oplus \mathcal{E}^{\perp}$ is such that
    $Z = P|_{\mathcal{E}}$ is weakly stable contraction and $U = P|_{\mathcal{E}^{\perp}}$ is unitary.
\end{thm}

\noindent A $\Gamma_n$-contraction $(S_1,\dots ,S_{n-1},P)$ admits
an analogous decomposition which is presented in the following
theorem.
 \begin{thm}\label{Foguel1}
        Let $\big( S_1, \dots, S_{n-1}, P \big)$ be a
        $\Gamma_n$-contraction on a Hilbert space $\mathcal{H}$ and let $\mathcal{E}$ and $\mathcal{E}^{\perp}$ . Then
    \begin{enumerate}
        \item $\mathcal{E}$ and $\mathcal{E}^{\perp}$ reduce $S_1, \dots, S_{n-1}$,
        \item $\left( S_1|_{\mathcal{E}}, \dots, S_{(n-1)}|_{\mathcal{E}},
        P|_{\mathcal{E}} \right)$ is a weakly stable $\Gamma_n$-contraction,
        \item $\left(S_1|_{\mathcal{E}^{\perp}}, \dots, S_{(n-1)}
        |_{\mathcal{E}^{\perp}}, P|_{\mathcal{E}^{\perp}}\right)$ is a $\Gamma_n$-unitary.
    \end{enumerate}
 \end{thm}
\begin{proof}
    The proof goes through if we show that $\mathcal{E}$ is a reducing subspace for $S_1, \dots, S_{n-1}$.
    Set $$\mathcal{E}^{\prime} = \{x\in \mathcal{H} : \langle P^nx,x \rangle
    \rightarrow 0 \text{ as } n \rightarrow \infty\}\,,$$ $$\mathcal{E}_{*} =
    \{x\in \mathcal{H} : \langle P^{*n}x,y \rangle \rightarrow 0 \text{ as }
    n \rightarrow \infty, \text{ for all } y \in \mathcal{H}\}$$ and
    $$\mathcal{E}_{*}^{\prime} = \{x\in \mathcal{H} : \langle P^{*n}x,x \rangle
    \\ \rightarrow 0 \text{ as } n \rightarrow \infty\}.$$
    Then following the proof of Theorem 7.3 in \cite{K:C}, we have
    the following equality:
    \[
    \mathcal{E} = \mathcal{E}^{\prime} = \mathcal{E}_{*} = \mathcal{E}_{*}^{\prime}.
    \]
    Now for any $ y \in \mathcal{H} $ and  for all $i = 1, \dots, n-1$,
    \[\langle P^nS_ix, y \rangle = \langle S_i P^nx, y \rangle =
    \langle P^nx, S_i^*y \rangle  \rightarrow 0 \text{ as } n \rightarrow \infty. \]
     Similarly, for any $y\in\mathcal{H}$, $$\langle P^{*n}
     S_i^*x,y \rangle = \langle S_i^*P^{*n}x,y \rangle =
     \langle P^{*n}x,S_iy \rangle \rightarrow 0 \text{ as }
     n \rightarrow \infty \,.$$ Therefore, $\mathcal{E}$
     is reducing subspace for $S_1, \dots, S_{n-1}$ and the proof
     is complete.
   \end{proof}

\section{Structure theorems for $\mathbb
E$-contractions}\label{tetra-block}

In the previous section we presented few decomposition theorems
for $\Gamma_n$-contractions. In this section, we show that similar
theorems can be obtained for $\mathbb E$-contractions. We begin
with the analogue of Theorem \ref{thmlv1} in the tetrablock
setting.

\begin{thm}\label{cd11}
    Let $(A,B,P)$ be a c.n.u $\mathbb E$-contraction on a Hilbert
    space $\mathcal{H}$. Let $\mathcal{H}_1$ be the maximal subspace
    of $\mathcal{H}$ which reduces $P$ and on which $P$ is isometry.
    Let $\mathcal{H}_2 = \mathcal{H}\ominus \mathcal{H}_1$. If either
    $A^*, B^*$ commute with $P$ or $\mathcal{H}_1$ is the maximal
    invariant subspace for $P$ on which $P$ is isometry, then
    \begin{enumerate}
        \item $\mathcal{H}_1$, $\mathcal{H}_2$ reduce $A$ and $B$;
        \item $\left(A|_{\mathcal{H}_1}, B|_{\mathcal{H}_1},P|_{\mathcal{H}_1}\right)$ is a c.n.u $\mathbb E$-isometry;
        \item $\left(A|_{\mathcal{H}_2},B|_{\mathcal{H}_2},P|_{\mathcal{H}_2}\right)$ is a c.n.i $\mathbb E$-contraction.
    \end{enumerate}
    The subspaces $\mathcal{H}_1$ or $\mathcal{H}_2$ may equal to the trivial subspace $\{0\}$.
\end{thm}
\begin{proof}
    \textbf{Case I.} Let $A^*$ and $B^*$  commute with $P$. If $P$ is a c.n.i contraction
    then $\mathcal H_1=\{0\}$ and if $P$ is a c.n.u isometry then $\mathcal
    H=\mathcal H_1$ and so $\mathcal H_2=\{0\}$. In such cases the
    theorem is trivial. Suppose $P$ is neither a c.n.u isometry nor
    a c.n.i contraction. With respect to the decomposition $\mathcal H=\mathcal H_1\oplus
    \mathcal H_2$, let
    \[
    A=
    \begin{bmatrix}
    A_{11}&A_{12}\\
    A_{21}&A_{22}
    \end{bmatrix}\,,\,
    B=
    \begin{bmatrix}
    B_{11}&B_{12}\\
    B_{21}&B_{22}
    \end{bmatrix}
    \text{ and } P=
    \begin{bmatrix}
    P_1&0\\
    0&P_2
    \end{bmatrix}
    \]
    so that $P_1$ is a shift operator and $P_2$ is c.n.i.
    Since $(A, B, P)$ is an $\mathbb E$-contraction on a Hilbert space
    $\mathcal{H}$, so there exist two unique operators $F_1$ and $F_2$ on $\mathcal{D}_P$ such that
    \[
    A-B^*P = D_PF_1D_P \quad \text{ and }\quad B-A^*P = D_PF_2D_P.
    \]
    With respect to the decomposition $\mathcal{D}_P = \mathcal{D}_{P_1} \oplus \mathcal{D}_{P_2} = \{0\}\oplus\mathcal{D}_{P_2}$, let
    \[ F_i=
    \begin{bmatrix}
    0 & 0\\
    0 & F_{i22}
    \end{bmatrix}
    \text{ for } i = 1, 2.
    \]
    Then from $A-B^*P = D_PF_1D_P$, we have
    \begin{align}
    A_{11}&=B_{11}^*P_1      & A_{12}&=B_{21}^*P_2\,, \label{eqn:2.1} \\
    A_{21}&=B_{12}^*P_1    & A_{22} - B_{22}^*P_2&=D_{P_2}F_{122}D_{P_2}\,. \label{eqn:2.2}
    \end{align}
    Similarly from $B-A^*P = D_PF_2D_P$, we have
    \begin{align}
    B_{11}&=A_{11}^*P_1      & B_{12}&=A_{21}^*P_2\,, \label{eqn:2.3} \\
    B_{21}&=A_{12}^*P_1    & B_{22} - A_{22}^*P_2&=D_{P_2}F_{222}D_{P_2}\,. \label{eqn:2.4}
    \end{align}
    Since $\|B\|\leq 1$, so $\|B_{11}\|\leq 1$. Therefore, by
    part-(3) of Theorem \ref{thm:ti}, $(A_{11},B_{11},P_1)$ is a c.n.u $\mathbb E$-isometry.\\
    Since $AP = PA$ and $BP = PB$, so we have
    \begin{align}
    A_{11}P_1&=P_1A_{11}    & A_{12}P_2=P_1A_{12}\,, \label{eqn:2.5} \\
    A_{21}P_1&=P_2A_{21}    & A_{22}P_2=P_2A_{22}\,, \label{eqn:2.6} \\
    B_{11}P_1&=P_1B_{11}    & B_{12}P_2=P_1B_{12}\,, \label{eqn:2.7} \\
    B_{21}P_1&=P_2B_{21}    & B_{22}P_2=P_2B_{22}\,. \label{eqn:2.8}
    \end{align}
    From the first equation in (\ref{eqn:2.6}) we have range
    of $A_{21}$ is invariant under $P_2$. Again from $AP^*=P^*A$ we have
    \[
    A_{21}P_1^* = P_2^*A_{21},
    \]
    that is range of $A_{21}$ is invariant under $P_2^*$ also.
    Therefore, range of $A_{21}$ is reducing subspace for $P_2$.
    Now from the first equation in (\ref{eqn:2.6}) and the second equation in (\ref{eqn:2.3}) we have
    \[
    P_2^*P_2A_{21}  = P_2^*A_{21}P_1 = B_{12}^*P_1 = A_{21}.
    \]
    This shows that $P_2$ is isometry on range of $A_{21}$. But $P_2$ is completely non-isometry.
    Therefore, we must have $A_{21}=0.$ Now from (\ref{eqn:2.3}), $B_{12} = 0$.
    Similarly, we can prove that $B_{21} = 0$. Now from (\ref{eqn:2.1}), $A_{12} = 0$.
    Thus with respect to the decomposition $\mathcal{H} =\mathcal{H}_1 \oplus \mathcal{H}_2$
    \[
    A=
    \begin{bmatrix}
    A_{11}&0\\
    0&A_{22}
    \end{bmatrix}\,,\,
    B=
    \begin{bmatrix}
    B_{11}&0\\
    0&B_{22}
    \end{bmatrix}.
    \]
    So, $\mathcal{H}_1$ and $\mathcal{H}_2$ reduce $A$ and $B$. Now $(A_{22}, B_{22},P_2)$
    is the restriction of the $\mathbb E$-contraction $(A, B, P)$ to the reducing
    subspace $\mathcal{H}_2$. Therefore, $(A_{22}, B_{22}, P_2)$ is an $\mathbb E$-contraction.
    Since $P_2$ is c.n.i, $(A_{22}, B_{22}, P_2)$ is a c.n.i $\mathbb E$-contraction.  \\

   \noindent  \textbf{Case II.} Suppose $\mathcal{H}_1$ is also the maximal
   invariant subspace for $P$ on which $P$ is isometry. Now from the first
   equation in (\ref{eqn:2.6}) and the second equation in (\ref{eqn:2.3}) we have range of $A_{21}$ is invariant under $P_2$ and
    \[
    P_2^*P_2A_{21}  = P_2^*A_{21}P_1 = B_{12}^*P_1 = A_{21}.
    \]
    This shows that $P_2$ is isometry on range of $A_{21}$. Therefore,
    we must have $A_{21}=0.$ Now from (\ref{eqn:2.3}), $B_{12} = 0$.
    Similarly, we can prove that $B_{21} = 0$. Now from (\ref{eqn:2.1}), $A_{12} = 0$. This completes the proof.
\end{proof}

The above theorem may not be true if we drop the assumptions that
either $A^*$ and $B^*$ commute with $P$ or $\mathcal{H}_1$ is also
the maximal invariant subspace for $P$ on which $P$ is an
isometry. Before going to present a counterexample we shall recall
some useful facts from the literature.

\begin{lem}[\cite{S:S}, Theorem 4.3]\label{countlem}
    If $(s, p) \in \Gamma$ then $\left(\dfrac{s}{2}, \dfrac{s}{2},p\right) \in \overline{\mathbb E}$.
\end{lem}
An operator-analogue of the previous lemma follows immediately.
\begin{lem}\label{Econtrac}
    If $(S, P)$ is a $\Gamma$-contraction then $\left(\dfrac{S}{2}, \dfrac{S}{2}, P\right)$ is an $\mathbb E$-contraction.
\end{lem}
\begin{proof}
    Let $g$ be the map from $\Gamma$ to $\mathbb E$ that maps $(s, p)$ to $(\dfrac{s}{2}, \dfrac{s}{2}, p)$.
    Then for any holomorphic polynomial $f$ in three variables we have
    \[
    \left\| f\left(\dfrac{S}{2}, \dfrac{S}{2}, P\right)\right\| =
    \|f\circ g(S, P)\| \leq \|f\circ g\|_{\infty, \Gamma} = \|f\|_{\infty, g(\Gamma)} \leq \|f\|_{\infty, \overline{\mathbb E}}\,.
    \]
    Then by Lemma \ref{lemTB} $\left(\dfrac{S}{2}, \dfrac{S}{2}, P\right)$ is an $\mathbb E$-contraction.
\end{proof}

\noindent We now present a counter example.

\begin{eg}\label{example12}
    We consider $\mathcal{H} = \ell^2$, where
    \[
    \ell^2 = \left\{\{x_n\}_n : x_n \in \mathbb{C} \text{ and } \sum_{n=1}^{\infty}|x_n|^2 < \infty\right\}.
    \]
    Now consider an operator $P : \ell^2 \rightarrow \ell^2$ defined by
    \[
    P(x_1, x_2, x_3, \dots) = (0, \frac{x_1}{2}, x_2, x_3, \dots).
    \]
    Then by Ando's dilation ( see \cite{Ando}) $(2P, P^2)$ is a $\Gamma$-contraction on $\ell^2$.
    Therefore, by Lemma \ref{Econtrac} $(P, P, P^2)$ is an $\mathbb E$-contraction on $\ell^2$. It is clear that
    \[
    P^{*2}P^2(x_1, x_2, x_3, \dots) = (\frac{x_1}{4}, x_2, x_3, \dots).
    \]
    Suppose $H = \{(0, x_1, x_2, \dots) : x_i \in \mathbb{C}\}$. Then clearly $H$ is
    the maximal invariant subspace for $P^2$ on which $P^2$ is isometry.
    One can easily check that $P^*P^2 \neq P^2P^*$. Suppose $H_1 (\subset H)$
    is the maximal reducing subspace for $P^2$ on which $P^2$ is isometry.\\

    \noindent \textbf{Claim:} $H_1 = \{(0, x_1, 0, x_2, 0, x_3, \dots) : x_i \in \mathbb{C}\}$.\\
    \noindent \textbf{Proof of Claim.} It is evident that $\{(0, x_1, 0, x_2, 0, x_3, \dots) : x_i \in \mathbb{C}\}
    \subseteq H_1$. Suppose $(0, \dots, 0, \underbrace{x_{2n+1}}_{\text{(2n+1)th position}}, 0, \dots)
    \in H_1$. Since $H_1$ is reducing subspace for $P^2$ so $P^{*2n}(0, \dots, 0, x_{2n+1}, 0, \dots)
    \in H_1$ i.e., $(\frac{x_{2n+1}}{2}, 0, 0, \dots) \in H$. Therefore, we must have $x_{2n+1} = 0$.
    This completes the proof of claim. \\

    It is now clear that $H_1$ is not a reducing subspace for $P$.
\end{eg}

Here is an analogue of Kubrusly-type decomposition (Theorem
\ref{Kdecom}) for $\mathbb E$-contractions.

\begin{thm}\label{mainthm1}
    Let $(A,B,P)$ be an $\mathbb E$-contraction on a Hilbert space
    $\mathcal{H}$. If $\mathcal{T} = \mathcal{T}^2$, then
    \begin{enumerate}
        \item $Ker(\mathcal{T})$, $Ker(I-\mathcal{T})\cap Ker(\mathcal{T}_*)$
        and $Ker(I-\mathcal{T})\cap Ker(I-\mathcal{T}_*)$ reduce $A \text{ and } B$
        \item $\left(A|_{Ker(\mathcal{T})},B|_{Ker(\mathcal{T})},P|_{Ker(\mathcal{T})}\right)$
        is a strongly stable $\mathbb E$-contraction,
        \item $\left(A|_{Ker(I-\mathcal{T})\cap Ker(\mathcal{T}_*)},B|_{Ker(I-\mathcal{T})
        \cap Ker(\mathcal{T}_*)},P|_{Ker(I-\mathcal{T})\cap Ker(\mathcal{T}_*)}\right)$
        is a pure $\mathbb E$-isometry on $Ker(I-\mathcal{T})\cap Ker(\mathcal{T}_*) $,
        \item $\left(A|_{Ker(I-\mathcal{T})\cap Ker(I-\mathcal{T}_*)},
        B|_{Ker(I-\mathcal{T})\cap Ker(I-\mathcal{T}_*)},P|_{Ker(I-\mathcal{T})
        \cap Ker(I-\mathcal{T}_*)}\right)$ is an $\mathbb E$-unitary.
    \end{enumerate}
    Moreover, if $\mathcal{T} = \mathcal{T}^2$ and $\mathcal{T}_* = \mathcal{T}_{*}^2$, then
    \begin{enumerate}
        \item[$(1')$] $\left(A|_{Ker(\mathcal{T})\cap Ker(\mathcal{T}_{*})},
        B|_{Ker(\mathcal{T})\cap Ker(\mathcal{T}_{*})},P|_{Ker(\mathcal{T})\cap Ker(\mathcal{T}_{*})}\right)$
        is a $\mathcal{C}_{00}$  $\mathbb E$-contraction,
        \item[$(2')$] $\left(A^*|_{Ker(\mathcal{T})\cap Ker(I-\mathcal{T}_{*})},
        B^*|_{Ker(\mathcal{T})\cap Ker(I-\mathcal{T}_{*})},P^*|_{Ker(\mathcal{T})
        \cap Ker(I-\mathcal{T}_{*})}\right)$ is a pure $\mathbb E$-isometry
        on $Ker(\mathcal{T})\cap Ker(I-\mathcal{T}_{*})$.
    \end{enumerate}
    Furthermore, if $\mathcal{T} = \mathcal{T}_{*}$, then
    \begin{enumerate}
        \item[$(1'')$] $\left(A|_{Ker(\mathcal{T})},
        B|_{Ker(\mathcal{T})},P|_{Ker(\mathcal{T})}\right)$ is
        a $\mathcal{C}_{00}$  $\mathbb E$-contraction,
        \item[$(2'')$] $\left(A|_{Ker(I-\mathcal{T})},B|_{Ker(I-\mathcal{T})},
        P|_{Ker(I-\mathcal{T})}\right)$ is an $\mathbb E$-unitary.
    \end{enumerate}
\end{thm}

\begin{proof}
The proof is similar to the proof of Theorem \ref{mainthm} and we
skip it.
\end{proof}

The following theorem is an analogue of Theorem \ref{thm:11} for $\mathbb E$-contractions.

\begin{thm}\label{decomp:thm1}
    Let $\left( A, B , P \right)$ be an $\mathbb E$-contraction on a Hilbert space $\mathcal{H}$. Then
    \begin{enumerate}
        \item $ \overline {\text{Ran}}(I-P)$, $\text{Ker}(I-P)$ reduce $ A, B $;
        \item $\left( A, B, P \right)$ on $\overline{\text{Ran}}(I-P)$ is an $ \mathbb E $-contraction;
        \item $\left(A, B, P \right)$ on $\text{Ker}(I-P)$ is an $\mathbb E$-unitary.
    \end{enumerate}
\end{thm}

\begin{proof}
Similar to that of Theorem \ref{thm:11}.
\end{proof}

We present an analogue of Theorem \ref{NL1} for $\mathbb
E$-contractions.

\begin{thm}\label{cd41}
    Let $(A,B,P)$ be a c.n.u $\mathbb E$-contraction on $\mathcal{H}$ such
    that $A^*,B^*$ commute with $P$. Then with respect to the unique orthogonal decomposition
    $\mathcal{H} = \mathcal{H}_1^{\prime}\oplus\mathcal{H}_2^{\prime}
    $ as in Theorem \ref{NL1}, the following hold:

    \begin{enumerate}
        \item $\mathcal{H}_1^{\prime}$ is a maximal joint invariant subspace for $A, B, P$
        and that $\left(A|_{\mathcal{H}_1^{\prime}}, B|_{\mathcal{H}_1^{\prime}}, P|_{\mathcal{H}_1^{\prime}}\right)$ is c.n.u $\mathbb E$-isometry;
        \item $\mathcal{H}_2^{\prime}$ is a maximal joint invariant subspace for
        $A^*, B^*, P^*$ and that $\left(A^*|_{\mathcal{H}_2^{\prime}}, B^*|_{\mathcal{H}_2^{\prime}},
        P^*|_{\mathcal{H}_2^{\prime}} \right)$ is c.n.u $\mathbb E$-contraction.
    \end{enumerate}

\end{thm}

\begin{proof}
We may imitate the proof of Theorem \ref{cd4}.

\end{proof}

\noindent \textbf{Note.} We cannot drop the condition that $A^*,
B^*$ commute with $P$ in the hypothesis of Theorem \ref{cd51}.
Example \ref{example12} clearly shows that the conclusion of
Theorem \ref{cd51} may not be reached if $A^*, B^*$ do not commute
with $P$.\\

\noindent The next theorem that we are going to present is an analogue of Theorem \ref{NL2} for
$\mathbb E$-contractions and an extension of Theorem \ref{cd41}. So, naturally as in Theorem \ref{cd41}, we need to assume that $P$ commutes with $A^*,B^*$.

\begin{thm}\label{cd51}
    Let $(A,B,P)$ be a completely non-unitary $\mathbb E$-contraction
    on $\mathcal{H}$ such that $A^*P = PA^*$ and $B^*P = PB^*$. Then $\mathcal{H}$
    admits a unique orthogonal decomposition
    $
    \mathcal{H} = \mathcal{H}_1^{\circ} \oplus \mathcal{H}_2^{\circ} \oplus \mathcal{H}_3^{\circ},
    $
    such that
    \begin{enumerate}
        \item $\mathcal{H}_1^{\circ}$ and $\mathcal{H}_3^{\circ}$ are joint invariant subspaces for $A,\,B \text{ and } P$;
        \item $\mathcal{H}_2^{\circ}$ is a joint invariant subspace for $A^*,\,B^* \text{ and } P^*$\,;
        \item $\left( A^*|_{\mathcal{H}_1^{\circ}}, B^*|_{\mathcal{H}_1^{\circ}}, P^*|_{\mathcal{H}_1^{\circ}} \right)$
        is a pure $\mathbb E$-contraction;
        \item $\left( A|_{\mathcal{H}_2^{\circ}}, B|_{\mathcal{H}_2^{\circ}}, P|_{\mathcal{H}_2^{\circ}} \right)$
        is a c.n.u $\mathbb E$-contraction;
        \item $\left( A|_{\mathcal{H}_3^{\circ}}, B|_{\mathcal{H}_3^{\circ}}, P|_{\mathcal{H}_3^{\circ}} \right)$
        is a c.n.u $\mathbb E$-isometry.
    \end{enumerate}
\end{thm}

\begin{proof}
One can prove it easily if follows the proof of Theorem \ref{cd5}.
\end{proof}

We present an analogue of Foguel's theorem (Theorem \ref{Foguel})
for an $\mathbb E$-contraction.

\begin{thm}\label{Foguel11}
    Let $\left( A, B , P \right)$ be an $\mathbb E$-contraction on a Hilbert space $\mathcal{H}$. Then
    \begin{enumerate}
        \item $\mathcal{Z}$ and $\mathcal{Z}^{\perp}$ reduce $A$ and
        $B$,
        \item $\left( A|_{\mathcal{Z}},B|_{\mathcal{Z}}, P|_{\mathcal{Z}} \right)$ is a weakly stable $\mathbb E$-contraction,
        \item $\left(A|_{\mathcal{Z}^{\perp}},B|_{\mathcal{Z}^{\perp}}, P|_{\mathcal{Z}^{\perp}}\right)$ is an $\mathbb E$-unitary.
    \end{enumerate}
\end{thm}

\begin{proof}
The proof is similar to that of Theorem \ref{Foguel1}.
\end{proof}


\end{document}